\documentclass{article}
\usepackage[utf8]{inputenc}
\usepackage{tabularx}
\usepackage{amssymb}
\usepackage{amsmath}
\usepackage{theorem}
\usepackage{epsfig}
\usepackage{verbatim}
\usepackage{graphicx}
\usepackage{subfigure}
\usepackage{cancel}
\usepackage{epstopdf}
\usepackage{mathrsfs}  
\usepackage{a4wide}
\usepackage{pdflscape}

\usepackage{color}

\usepackage{amssymb}
\usepackage{amsmath}
\usepackage{theorem}
\usepackage{epsfig}
\usepackage{verbatim}
\usepackage{graphicx}
\usepackage{subfigure}
\usepackage{epstopdf}
\usepackage{hyperref}

\usepackage{color}

\newtheorem{theorem}{Theorem}[section]
\newtheorem{lemma}[theorem]{Lemma}

\newtheorem{corollary}[theorem]{Corollary}

\newtheorem{rem}[theorem]{Remark}

\makeatletter
\newcommand*{\rom}[1]{\expandafter\@slowromancap\romannumeral #1@}
\makeatother

\newcommand{\supp}{{\rm supp}\thinspace}

\newenvironment{proof}{\begin{trivlist} \item[] {\em Proof:}}{\hfill $\Box$
                       \end{trivlist}}

\makeatletter
\renewcommand*\l@section{\@dottedtocline{1}{0em}{1.5em}}
\renewcommand*\l@subsection{\@dottedtocline{2}{1.5em}{2.3em}}
\renewcommand*\l@subsubsection{\@dottedtocline{3}{3.8em}{3.7em}}
\makeatother

\numberwithin{equation}{section}

\allowdisplaybreaks[4]

\begin{document}
\title{Regularity of solutions to the Muskat equation}
\author{Jia Shi}
\maketitle
\begin{abstract}
In this paper, we show that if a solution to the Muskat problem in the case of different densities and the same viscosity is 
 sufficiently smooth, then it must be analytic except at the points 
 where a turnover of the fluids happens.
\end{abstract}

\section{Introduction}
The Muskat problem is a free boundary problem studying the interface between fluids in the porous media \cite{MuskatPhysics}. It can also describe the Hele-Shaw cell \cite{shaw1898motion}. The density function $\rho$ follows the active scalar equation:
\begin{equation}\label{physicalequation}
 \frac{d\rho}{dt}+(v\cdot \nabla)\rho=0,   
\end{equation}
with 
\begin{align*}
\rho(x,t)=\left\{\begin{array}{ccc}
\rho_1 && x\in D_1(t),\\
\rho_2 && x\in D_2(t). \end{array}\right.
\end{align*}
Here $D_1(t)$ and $D_2(t)$ are open domains with $D_1(t)\cup D_2(t) \cup \partial D_{1}(t)=\mathbb{R}^2$. The velocity field $v$ in \eqref{physicalequation} satisfies Darcy's law:
\begin{equation}\label{Darcylaw}
    \frac{\mu}{\kappa}v=-\nabla p -(0,g\rho),
\end{equation}
and the incompressibility condition: 
\[
\nabla \cdot v=0,
\]
where $p$ is the pressure and $\mu$ is the viscosity.  $\kappa$, $g$ are the permeability constant and the gravity force.

We focus on the problem where two fluids have different densities $\rho_1, \rho_2$ and the same viscosity $\mu$. 

After scaling, the equation for the boundary $\partial D_{1}(t)$  in the periodic setting reads:

\begin{equation}\label{muskat equation}
\frac{\partial f_i}{\partial t}(\alpha,t)=\frac{\rho_2-\rho_1}{2}\int_{-\pi}^{\pi} \frac{\sin(f_1(\alpha)-f_1(\alpha-\beta))(\partial_{\alpha}f_i(\alpha)-\partial_{\alpha}f_i(\alpha-\beta))}{\cosh(f_2(\alpha)-f_2(\alpha-\beta))-\cos(f_{1}(\alpha)-f_{1}(\alpha-\beta))}d\beta,
\end{equation}
for $i=1,2$ (See \cite{Castro-Cordoba-Fefferman-Gancedo-LopezFernandez:rayleigh-taylor-breakdown}). Here $f(\alpha,t)=(f_1(\alpha,t),f_2(\alpha,t))$ is a parameterization of the boundary curve. $f(\alpha,t)-(\alpha,0)$ is periodic in $\alpha$. 

 Given an initial interface at time 0, \eqref{muskat equation} is divided into three regimes. When the interface is a graph and the heavier fluid is on the bottom as in Figure \ref{fig_turning over1}a, it is in a stable regime. When heavier fluid is above the boundary as in Figure \ref{fig_turning over1}b, it is in a stable regime when time flows backward. Thus, given any initial data, \eqref{muskat equation} can be solved for small negative time $t$. In both regimes, shown in Figures \ref{fig_turning over1}a, \ref{fig_turning over1}b, \eqref{muskat equation} can not be solved in the wrong direction unless the initial interface is real analytic. 
 The third regime, shown in Figure \ref{fig_turning over1}c, it highly unstable because the heavier fluid lies on top near point $S_1$ while the lighter fluid lies on top near point $S_2$. Note two turnover points $T_1$ and $T_2$ where the interface has a vertical tangent. For generic initial data in the turnover regime, \eqref{muskat equation} has no solutions either as time flows forward or backward. 
\begin{figure}[h!]
\begin{center}
\includegraphics[scale=0.8]{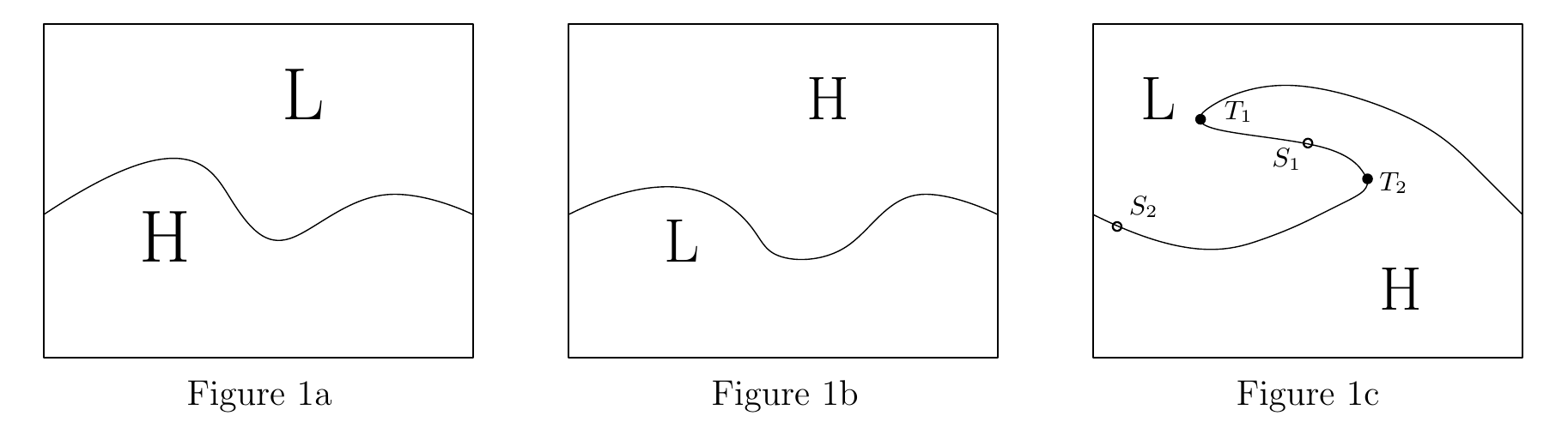}
\caption{\label{fig_turning over1}Three regimes of the Muskat equation.}
\end{center}
\end{figure}

In the third regime, there are several examples from the literature (eg. \cite{Castro-Cordoba-Fefferman-Gancedo-LopezFernandez:rayleigh-taylor-breakdown}, \cite{Castro-Cordoba-Fefferman-Gancede}, \cite{Cordoba-GomezSerrano-Zlatos:stability-shifting-muskat}, \cite{Cordoba-GomezSerrano-Zlatos:stability-shifting-muskat-II}), but they are all real analytic solutions.  Without the real analytic assumption, due to the spatially non-consistent parabolic behavior, the existence is usually false and the uniqueness is unknown. 
To address this gap, this paper studies to what extent the solution of $\eqref{muskat equation}$ is analytic.
 
Moreover, for the analytic solutions, one can prove an energy estimate on an analyticity region that shrinks when time increases. That energy estimate implies uniqueness in the class of analytic solutions. \cite{Castro-Cordoba-Fefferman-Gancedo-LopezFernandez:rayleigh-taylor-breakdown}. Therefore,  the investigation towards analyticity can serve as a first step to deal with the uniqueness.

We introduce a new way to prove that any sufficiently smooth solution is analytic except at the turnover points. Here is our main theorem:
 \begin{theorem}\label{thm1}
 Let $f(\alpha,t)=(f_1(\alpha,t),f_2(\alpha,t))\in C^{1}([-t_0,t_0],H^{6}[-\pi,\pi]\times H^{6}[-\pi,\pi]))$ be a solution of the Muskat equation \eqref{muskat equation} satisfying the arc-chord condition. If $\partial_{\alpha}f_1(\alpha_0,t)\neq 0$, and $-t_0< t<t_0$, then $f(\cdot,t)$ is analytic at $\alpha_0$.
 \end{theorem}

Our method concerning the analyticity is not limited to the Muskat problem. A simplified version of our method can be used to show the analyticity of the solution to a kind of non-local differential equations (see Section \ref{analyticitystationary}). This approach is new to our best knowledge.

In our forthcoming work \cite{muskatregulatiryturnover2}, we focus on the degenerate analyticity near the turnover points. The existence and uniqueness are crucially related to the way the real-analyticity degenerates at those points.
  Given an extra assumption, we have the following theorem in \cite{muskatregulatiryturnover2}:
 \begin{theorem}\label{muskatnear}
 Let $f(\alpha,t)=(f_1(\alpha,t),f_2(\alpha,t))\in C^{1}([-t_0,t_0],C^{100}([-\pi.\pi])$ be a solution of the Muskat equation \eqref{muskat equation}  with two turnover points. $Z_1(t)$, $Z_2(t)$ are values of $\alpha$ of these two turnover points. If we assume the solution satisfies the following three conditions:
\begin{align}\label{extracondition1}
\partial_{\alpha}^2f_1(Z_1(t),t)\neq 0,
\end{align}
\begin{align}\label{extracondition2}
\partial_{\alpha}f_1(\alpha,t) \neq 0 \text{  except at $Z_1(t)$, $Z_2(t)$}, 
\end{align}
and
\begin{align}\label{extracondition3}
(\frac{d Z_1}{dt}(t)+\frac{\rho_2-\rho_1}{2}p.v.\int_{-\pi}^{\pi}\frac{\sin(f_1(\alpha)-f_1(\alpha-\beta))}{\cosh(f_2(\alpha)-f_2(\alpha-\beta))-\cos(f_1(\alpha)-f_1(\alpha-\beta))}d\beta) \frac{\rho_2-\rho_1}{2}\partial_{\alpha}^2f_1(Z_1(t),t)<0,
\end{align}
then when $-t_0<t<t_0$, $f(\cdot,t)$ can be analytically extended to region $\Omega=\{x+iy|-\epsilon_1(t)+Z_{1}(t)\leq x\leq Z_{1}(t)+\epsilon_1(t),|y|\leq \epsilon_2(t)(x-Z_{1}(t))^2\}$.
\end{theorem}
\color{black} 
\subsection{Background}
In order to make the equation well-defined, the arc-chord condition is introduced, saying
\[
F(f)=|\frac{\beta^2}{\cosh(f_2(\alpha)-f_2(\alpha-\beta))-\cos(f_{1}(\alpha)-f_{1}(\alpha-\beta))}|
\]
is in $L^{\infty}$.

The Rayleigh-Taylor coefficient $\sigma$ is used to characterize the three regimes in Figure \ref{fig_turning over1} and is defined as 
\begin{equation}
\sigma=\frac{\rho_2-\rho_1}{2}\frac{\partial_{\alpha}f_{1}(\alpha,t)}{(\partial_{\alpha}f_1(\alpha,t))^2+(\partial_{\alpha}f_2(\alpha,t))^2}.
\end{equation}
$\sigma\geq 0$ is corresponding to the stable regime and $\sigma\leq 0$ the backward stable regime. When $\sigma$ changes sign, it is in the unstable regime.

In the stable regime (heavier liquid is below the lighter liquid), local well-posedness and the global well-posedness with constraints on the initial data have been widely studied, with the lowest space $H^{\frac{3}{2}}$   
(\cite{Yifahuaimuskatlocalexistence},\cite{YI2003442},\cite{Ambrose2004WellposednessOT},\cite{MichaelRusselHowisonmuskat},\cite{CordobaGancedocontourdynamicsmuskat},\cite{CordobaGancedomuskatmaximum},\cite{Constantin2013OnTG},\cite{CordobaCordobaGancedomuskatexistence},\cite{CHENG201632},\cite{constantincordobagancedostrainmuskatexistence},\cite{CONSTANTIN20171041},\cite{stephenmuskatexistence},\cite{Deng2016OnTT},\cite{DiegoOmarmuskatexistence},\cite{ Matioc:local-existence-muskat-hs},\cite{NguyenBenoitparadifferentialmuskat},\cite{alazardomarparalinearizationmuskat},\cite{abels_matioc_2022},\cite{muskatc12021chenquocxu},\cite{NGUYEN2022108122},\cite{alazardthomascritialspacemuskat},\cite{Alazard-Nguyennonlipshitzmuskat}, \cite{Alazard2020EndpointST}). The existence of self-similar solutions has also been proved \cite{2021selfsimilar}. Interesting readers can see \cite{2021selfsimilar} and \cite{muskatc12021chenquocxu} for detailed reviews. Due to the parabolic behavior, instant analyticity has been proved in the stable regime. Castro--Córdoba--Fefferman--Gancedo--López-Fernández \cite{Castro-Cordoba-Fefferman-Gancedo-LopezFernandez:rayleigh-taylor-breakdown} proved the $H^4$ solutions become instantly analytic if the solutions remain to be in the stable region for a short time. In \cite{Matioc:local-existence-muskat-hs}, also in the stable region,  Matioc improved the instant analyticity to $H^{s}$, where  $s\in (\frac{3}{2},3)$. In \cite{GANCEDO2019552}, 
Gancedo--García-Juárez--Patel--Strain showed that in the stable regime, a medium size initial data in $\dot{\mathcal{F}^{1,1}}\cap L^2$ with $\|f\|_{\mathcal{F}^{1,1}}=\int|\zeta||\hat{f}(\zeta)|d\zeta$ becomes instantly analytic. Their result also covers the different viscosities case and the 3D case. 

When the heavier liquid is above the lighter liquid,  the equation is ill-posed when time flows forward \cite{CordobaGancedocontourdynamicsmuskat}.

A solution that starts from a stable regime and develops turnover points was first discovered in \cite{Castro-Cordoba-Fefferman-Gancedo-LopezFernandez:rayleigh-taylor-breakdown}. That solution still exists for a short time after turnover due to the analyticity when the turnover happens. Moreover, breakdown of smoothness can happen \cite{Castro-Cordoba-Fefferman-Gancede}. There are also examples where the solutions transform from stable to unstable and go back to stable \cite{Cordoba-GomezSerrano-Zlatos:stability-shifting-muskat} and vice versa \cite{Cordoba-GomezSerrano-Zlatos:stability-shifting-muskat-II} .



\color{black}


\subsection{The outline of the proof of Theorem \ref{thm1}}

     Inspired by the instant analyticity results in the stable case \cite{Castro-Cordoba-Fefferman-Gancedo-LopezFernandez:rayleigh-taylor-breakdown},  \cite{Matioc:local-existence-muskat-hs}, \cite{GANCEDO2019552}, our first idea is localization. If locally the  lighter liquid is over the heavier one, we let the time go forward, and if locally the heavier one is over the lighter one, we let the time go backward. 
  
  Since it leads to lots of difficulties by the standard method due to the localization, we use a new idea to prove analyticity except at turnover points. The idea is to make a $C^1$ continuation of the parametrized interface $\alpha\to(f_1(\alpha,t),f_2(\alpha,t))$ to complex $\alpha$ and then prove the $C^1$ continuation satisfies the Cauchy-Riemann equation. To do so, we break the complex region into curves $\alpha+ic(\alpha)\gamma t$ with $\gamma \in [-1,1]$. On each such curve, we solve an equation for $(f_1,f_2).$ We then show when $\gamma$ varies, that our solutions on the curve fit together into an $C^1$ function of $\alpha+i\beta$. Finally, we prove that $C^1$ function satisfies the Cauchy-Riemann equation, thus producing the desired analytic continuation.
  

In Section \ref{localizefar}, we define a cut off function $\lambda(\alpha)$ and focus on $f^c(\alpha,t)=\lambda(\alpha)f(\alpha,t)$. We then localize the equation such that the modified $R$-$T$ condition has a fixed sign.   In order to make use of the sign, if the sign is positive, we let the time go forward. If the sign is negative, we let the time go backward.

In Section \ref{extenedequationfar}, we introduce $c(\alpha)$ with $\supp c(\alpha)\subset\{\alpha|\lambda(\alpha)=1\}$. With the assumption that $f^c(\alpha,t)$ is analytic in domain $D_A=\{\alpha+i\beta|-c(\alpha)t\leq \beta\leq c(\alpha)t\}$, we derive the equation on the curve $\{(\alpha+ic(\alpha)\gamma t)|\alpha\in[-\pi,\pi]\}$ for fixed $\gamma \in [-1,1]$. Then we obtain the equation 
\begin{equation}\label{Tequation0intro}
\frac{d}{dt}z(\alpha,\gamma,t)=T(z(\alpha,\gamma,t),t),
\end{equation}
with $z(\alpha,\gamma,0)=f^c(\alpha,0)$. The analyticity assumption on $f^c$ is dropped after we get \eqref{Tequation0intro}.  
\begin{figure}[h!]
\begin{center}
\includegraphics[scale=0.8]{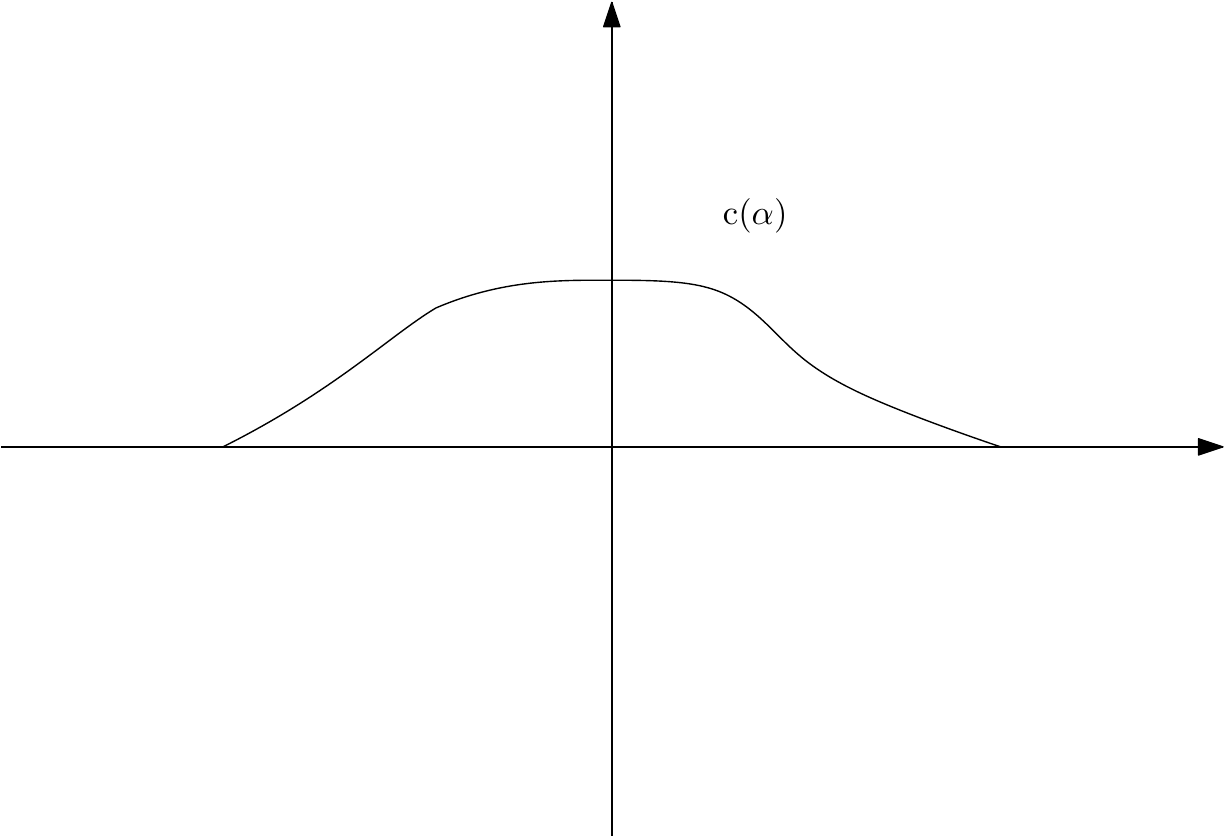}
\caption{\label{fig_Da}The curve $c(\alpha)$ for Theorem \ref{thm1}.}
\end{center}
\end{figure}

In Section \ref{existencefar},  for each fixed $\gamma$, we use the energy estimate and the Galerkin method to show the existence of the solution $z(\alpha,\gamma,t)$. The main term is controlled by G\r{a}rding's inequality, where we use a lemma from \cite{Castro-Cordoba-Fefferman-Gancede}. This part is similar as in \cite{Castro-Cordoba-Fefferman-Gancede} and \cite{Castro-Cordoba-Fefferman-Gancedo-LopezFernandez:rayleigh-taylor-breakdown}.

In Section \ref{uniquenessfar}, Section \ref{continuityfar}, and  Section \ref{diffrenfar}, we verify that the $z(\alpha,0,t)$ coincides with the $f^c(\alpha,t)$ and that $z(\alpha,\gamma,t)$ is also smooth enough with respect to $\gamma$.

In Section \ref{analyticityfar}, we derive some lemmas about the Cauchy-Riemann operator and use those lemmas to show analyticity of $z(\alpha,\frac{\beta}{c(\alpha )t}\gamma,t)$ by checking it satisfies the Cauchy-Riemann equations.

\begin{rem}
 In \cite{Castro-Cordoba-Fefferman-Gancedo-LopezFernandez:rayleigh-taylor-breakdown}, the analyticity domain can be chosen as a strip, and the analyticity follows directly from existence. Since our $c(\alpha)$ is supported in a small region, we do not have such good behavior.
\end{rem}

  

\section{Notation}
In the paper we will use the following notations:

$\delta$: a sufficiently small number.

$\lambda(\alpha)$: $\lambda(\alpha)\geq 0$ and in $C^{100}(-\infty,\infty)$, satisfying
\begin{align*}
    \lambda(\alpha)=\begin{cases}1&|\alpha|\leq \delta,\\0&|\alpha|\geq 2\delta.
\end{cases}
\end{align*}

$\delta_c$:  sufficiently small number depending on $\delta$.

$c(\alpha)$:
\begin{align*}
\Bigg\{\begin{array}{cc}
          c(\alpha)=\delta_c ,  \text{ when }|\alpha| \leq \frac{\delta}{32},\\
          \supp c(\alpha) \subset [-\frac{\delta}{8},\frac{\delta}{8}],\\
          c(\alpha)\geq 0, c(\alpha) \in C^{100}(-\infty,\infty), \|c(\alpha)\|_{C^{100}(-\infty,\infty)}\leq \delta.
    \end{array}
    \end{align*}

$f(\alpha,t)=(f_1(\alpha,t),f_2(\alpha,t))$: the original solution of the Muskat equation.

$f^c(\alpha,t)$, $\tilde{f}(\alpha,t)$:
\[
f^c(\alpha,t)=\lambda(\alpha)f(\alpha,t),
\]
\[
\tilde{f}(\alpha,t)=(1-\lambda(\alpha))f(\alpha,t).
\]

$t_0$: the original solution exists when $t\in[-t_0,t_0]$.

$D_A$: $D_A=\{(\alpha+i\beta)|-\infty < \alpha< \infty, -c(\alpha)t\leq \beta \leq c(\alpha)t\}.$

For any vector function $z =(z_1,z_2)\in H^{k}$: $z_1\in H^{k}$ and $z_2\in H^{k}$. 



\section{The localization}\label{localizefar}
This step is to localize the equation such that the $R$-$T$ coefficient has a fixed sign. Without loss of generality, we study the behavior at origin and let $\frac{\rho_2-\rho_1}{2}=1$.
Let $\lambda(\alpha) \in C^{100}(-\infty,\infty)$ satisfying $\lambda(\alpha)\geq 0$ and
\begin{align*}
    \lambda(\alpha)=\begin{cases}1&|\alpha|\leq \delta,\\0&|\alpha|\geq 2\delta,
\end{cases}
\end{align*} and $f^{c}(\alpha,t)=f(\alpha,t)\lambda(\alpha)$, $ \tilde{f}(\alpha,t)=f(\alpha,t)(1-\lambda(\alpha))$. Here $\delta$ is a sufficiently small number such that when $\alpha \in[-2\delta,2\delta]$, $\partial_{\alpha}f_1(\alpha,0)$ has a fixed sign. Without loss of generality, we assume
\begin{equation}\label{RTlocal}
    \partial_{\alpha}f_1(\alpha,0)>0.
\end{equation}

Then we have
\begin{align}\label{fequation}
\frac{\partial f^c_{\mu}}{\partial t}=\lambda(\alpha)\int_{-\pi}^{\pi} \frac{\sin(f_1^{c}(\alpha)-f_1^c(\beta)+\tilde{f_1}(\alpha)-\tilde{f_1}(\beta))(\partial_{\alpha}f^c_{\mu}(\alpha)-\partial_{\alpha}f^c_{\mu}(\beta))}{\cosh(f_2^c(\alpha)-f_2^c(\beta)+\tilde{f}_2(\alpha)-\tilde{f_2}(\beta))-\cos(f_{1}^c(\alpha)-f_{1}^c(\beta)+\tilde{f_1}(\alpha)-\tilde{f_1}(\beta)))}d\beta+\\
\lambda(\alpha)\int_{-\pi}^{\pi} \frac{\sin(f_1^{c}(\alpha)-f_1^c(\beta)+\tilde{f_1}(\alpha)-\tilde{f_1}(\beta))(\partial_{\alpha}\tilde{f}_{\mu}(\alpha)-\partial_{\alpha}\tilde{f}_{\mu}(\beta))}{\cosh(f_2^c(\alpha)-f_2^c(\beta)+\tilde{f}_2(\alpha)-\tilde{f_2}(\beta))-\cos(f_{1}^c(\alpha)-f_{1}^c(\beta)+\tilde{f_1}(\alpha)-\tilde{f_1}(\beta)))}d\beta.\nonumber
\end{align}
 We have $f^c\in C^1([0,t_0], (H^6(\mathbb{T}))^2)$, $\tilde{f}-(\alpha,0)\in C^1([0,t_0], (H^6(\mathbb{T}))^2)$. Here $\mathbb{T}$ is the torus of 2$\pi$.

\section{The equation on the complex plane}\label{extenedequationfar}
\subsection{Change the contour}
Let $c(\alpha)$ satisfy
\begin{align}\label{deltac}
\Bigg\{\begin{array}{cc}
          c(\alpha)=\delta_c ,  \text{ when }|\alpha| \leq \frac{\delta}{32},\\
          \supp c(\alpha) \subset [-\frac{\delta}{8},\frac{\delta}{8}],\\
          c(\alpha)\geq 0, c(\alpha) \in C^{100}(-\infty,\infty), \|c(\alpha)\|_{C^{100}(-\infty,\infty)}\leq \delta.
    \end{array}
\end{align}
Here $c(\alpha)$ is defined such that $\tilde{f}, \lambda$ can be analytically extended to the complex domain $D_A=\{(\alpha+i\beta)|-\infty < \alpha< \infty, -c(\alpha)t\leq \beta \leq c(\alpha)t\}$ and satisfy 
\begin{align}\label{tildefana}
\tilde{f}(\alpha+ic(\alpha)\gamma t,t)=\tilde{f}(\alpha,t),
\end{align}

and
\begin{align}\label{lambdaana}
\lambda(\alpha+ic(\alpha)\gamma t)=\lambda(\alpha),
\end{align}
for any $\gamma\in[-1,1].$

Now we assume $f^c$ is also analytic in this complex domain $D_A$. For any fixed $\gamma \in [-1,1]$, we want to find the new equation on the contour $\{\alpha+ic(\alpha)\gamma t|\alpha\in[-\pi,\pi]\}$. Let $\alpha_{\gamma}^t=\alpha+ic(\alpha)\gamma t$. We have
\begin{align}\label{changecontourequa}
&\frac{d f_{\mu}^c(\alpha_{\gamma}^t, t)}{dt}=ic(\alpha)\gamma(\partial_{\alpha}f_{\mu}^c)(\alpha_{\gamma}^t, t)+(\partial_{t}f_{\mu}^c)(\alpha_{\gamma}^t, t)\\\nonumber
&=ic(\alpha)\gamma(\partial_{\alpha}f_{\mu}^c)(\alpha_{\gamma}^t, t)\\\nonumber
&+\lambda(\alpha_{\gamma}^t)\int_{-\pi}^{\pi} \frac{\sin(f_1^{c}(\alpha_{\gamma}^t,t)-f_1^c(\beta_{\gamma}^t,t)+\tilde{f_1}(\alpha_{\gamma}^t,t)-\tilde{f_1}(\beta_{\gamma}^t,t))((\partial_{\alpha}f^{c}_{\mu})(\alpha_{\gamma}^t,t)-(\partial_{\beta}f^{c}_{\mu})(\beta_{\gamma}^t,t))}{\cosh(f_2^{c}(\alpha_{\gamma}^t,t)-f_2^c(\beta_{\gamma}^t,t)+\tilde{f_2}(\alpha_{\gamma}^t,t)-\tilde{f_2}(\beta_{\gamma}^t,t))-\cos(f_1^{c}(\alpha_{\gamma}^t,t)-f_1^c(\beta_{\gamma}^t,t)+\tilde{f_1}(\alpha_{\gamma}^t,t)-\tilde{f_1}(\beta_{\gamma}^t,t))}\\\nonumber
&(1+ic'(\beta)\gamma t)d\beta\\\nonumber
&+\lambda(\alpha_{\gamma}^t)\int_{-\pi}^{\pi} \frac{\sin(f_1^{c}(\alpha_{\gamma}^t,t)-f_1^c(\beta_{\gamma}^t,t)+\tilde{f_1}(\alpha_{\gamma}^t,t)-\tilde{f_1}(\beta_{\gamma}^t,t))((\partial_{\alpha}\tilde{f_{\mu}})(\alpha_{\gamma}^t,t)-(\partial_{\beta}\tilde{f_{\mu}})(\beta_{\gamma}^t,t))}{\cosh(f_2^{c}(\alpha_{\gamma}^t,t)-f_2^c(\beta_{\gamma}^t,t)+\tilde{f_2}(\alpha_{\gamma}^t,t)-\tilde{f_2}(\beta_{\gamma}^t,t))-\cos(f_1^{c}(\alpha_{\gamma}^t,t)-f_1^c(\beta_{\gamma}^t,t)+\tilde{f_1}(\alpha_{\gamma}^t,t)-\tilde{f_1}(\beta_{\gamma}^t,t))}\\\nonumber
&(1+ic'(\beta)\gamma t)d\beta.
\end{align}
\subsection{The equation on the curve}
Let $z(\alpha,\gamma,t)$ be the solution of the equation \eqref{changecontourequa} with initial data $z(\alpha,\gamma,0)=f^c(\alpha,0)$. Our motivation is to set $z(\alpha,\gamma,t)=f^{c}(\alpha+ic(\alpha)\gamma t,t).$ Since 
\[
(\partial_{\alpha}f^c_{\mu})(\alpha_{\gamma}^t,t)=\frac{\partial_{\alpha}(f^c_{\mu}(\alpha_{\gamma}^t,t))}{1+ic'(\alpha)\gamma t},
\]
we have
\begin{align}\label{equationz}
&\frac{d z_{\mu}(\alpha, \gamma,t)}{dt}=\frac{ic(\alpha)\gamma}{1+ic^{'}(\alpha)\gamma t}\partial_{\alpha}z_{\mu}(\alpha,\gamma,t)+\\\nonumber
&\lambda(\alpha_{\gamma}^{t})\int_{-\pi}^{\pi} \frac{\sin(z_{1}(\alpha,\gamma,t)-z_{1}(\beta,\gamma,t)+\tilde{f_1}(\alpha_{\gamma}^t,t)-\tilde{f_1}(\beta_{\gamma}^t,t))(\frac{\partial_{\alpha}z_{\mu}(\alpha,\gamma,t)}{1+ic'(\alpha)\gamma t}-\frac{\partial_{\beta}z_{\mu}(\beta,\gamma,t)}{1+ic'(\beta)\gamma t})(1+ic'(\beta)\gamma t)d\beta}{\cosh(z_{2}(\alpha,\gamma,t)-z_{2}(\beta,\gamma,t)+\tilde{f_2}(\alpha_{\gamma}^t,t)-\tilde{f_2}(\beta_{\gamma}^t,t))-\cos(z_{1}(\alpha,\gamma,t)-z_{1}(\beta,\gamma,t)+\tilde{f_1}(\alpha_{\gamma}^t,t)-\tilde{f_1}(\beta_{\gamma}^t,t))}\\\nonumber
&+\lambda(\alpha_{\gamma}^t)\int_{-\pi}^{\pi} \frac{\sin(z_{1}(\alpha,\gamma,t)-z_{1}(\beta,\gamma,t)+\tilde{f_1}(\alpha_{\gamma}^t,t)-\tilde{f_1}(\beta_{\gamma}^t,t))((\partial_{\alpha}\tilde{f_{\mu}})(\alpha_{\gamma}^t,t)-(\partial_{\beta}\tilde{f_{\mu}})(\beta_{\gamma}^t,t))(1+ic'(\beta)\gamma t)d\beta}{\cosh(z_{2}(\alpha,\gamma,t)-z_{2}(\beta,\gamma,t)+\tilde{f_2}(\alpha_{\gamma}^t,t)-\tilde{f_2}(\beta_{\gamma}^t,t))-\cos(z_{1}(\alpha,\gamma,t)-z_{1}(\beta,\gamma,t)+\tilde{f_1}(\alpha_{\gamma}^t,t)-\tilde{f_1}(\beta_{\gamma}^t,t))}
\end{align}
with
\[
z(\alpha,\gamma,0)=f(\alpha,0).
\]
We drop the  analyticity assumption of $f^c$ from now. Notice that $\tilde{f}$ and $\lambda$ can still be analytically extended to $D_A$ as in \eqref{tildefana} and \eqref{lambdaana}.
\section{The existence of z for fixed $\gamma$}\label{existencefar}
\subsection{Energy estimate}
We first assume $z$ is of finite Fourier modes here and do the energy estimate. The idea of the energy estimate is similar as in \cite{Castro-Cordoba-Fefferman-Gancedo-LopezFernandez:rayleigh-taylor-breakdown} and \cite{Castro-Cordoba-Fefferman-Gancede}.

Since $\tilde{f}(\alpha+ic(\alpha)\gamma t,t)=\tilde{f}(\alpha,t)$, $\lambda(\alpha+ic(\alpha)\gamma t)=\lambda(\alpha)$, we have
\begin{align}\label{zequationsimplify}
&\frac{d z_{\mu}(\alpha,\gamma,t)}{dt}=T(z)\\\nonumber
&=\frac{ic(\alpha)\gamma}{1+ic^{'}(\alpha)\gamma t}\partial_{\alpha}z_{\mu}(\alpha)+\\\nonumber
&\lambda(\alpha)\int_{-\pi}^{\pi} \frac{\sin(z_{1}(\alpha)-z_{1}(\beta)+\tilde{f_1}(\alpha)-\tilde{f_1}(\beta))(\frac{\partial_{\alpha}z_{\mu}(\alpha)}{1+ic'(\alpha)\gamma t}-\frac{\partial_{\beta}z_{\mu}(\beta)}{1+ic'(\beta)\gamma t})(1+ic'(\beta)\gamma t)}{\cosh(z_{2}(\alpha)-z_{2}(\beta)+\tilde{f_2}(\alpha)-\tilde{f_2}(\beta))-\cos(z_{1}(\alpha)-z_{1}(\beta)+\tilde{f_1}(\alpha)-\tilde{f_1}(\beta))}d\beta\\\nonumber
&+\lambda(\alpha)\int_{-\pi}^{\pi} \frac{\sin(z_{1}(\alpha)-z_{1}(\beta)+\tilde{f_1}(\alpha)-\tilde{f_1}(\beta))((\partial_{\alpha}\tilde{f_{\mu}})(\alpha)-(\partial_{\beta}\tilde{f_{\mu}})(\beta))(1+ic'(\beta)\gamma t)}{\cosh(z_{2}(\alpha)-z_{2}(\beta)+\tilde{f_2}(\alpha)-\tilde{f_2}(\beta))-\cos(z_{1}(\alpha)-z_{1}(\beta)+\tilde{f_1}(\alpha)-\tilde{f_1}(\beta))}d\beta.
\end{align}
Here we omit the dependency of $z$ on $\gamma$ and $t$, and the dependency of $\tilde{f}$ on $t$ for the sake of simplicity.
Let 
\[
U=(H^5(\mathbb{T}))^2,
\]
where $\mathbb{T}$ is the torus of length $2\pi$
and
\[
\|z\|_{Arc}=\sup_{\alpha\in[-2\delta,2\delta],\beta\in[-\pi,\pi]}\frac{(\alpha-\beta)^2}{|\cosh(z_{2}(\alpha)-z_{2}(\beta)+\tilde{f_1}(\alpha)-\tilde{f_1}(\beta))-\cos(z_{1}(\alpha)-z_{1}(\beta)+\tilde{f_1}(\alpha)-\tilde{f_1}(\beta))|}.
\]
For the $L^2$ norm, we have
\[
\|T(z)\|_{L^2[-\pi,\pi]}\lesssim\|T(z)\|_{L^\infty[-\pi,\pi]}\lesssim C(\|z\|_{U})\|z\|_{Arc}.
\]
Here $C$ is a bounded function depending on $\delta$, $\delta_c$ and $\|f\|_{C^1([0,t],(H^6[-\pi,\pi])^2)}$. We will keep using the same notation $C$ in the following proof.

Now we take 5th derivative and have
\begin{align*}
   &\partial_{\alpha}^5T(z)=\\
&=\frac{ic(\alpha)\gamma}{1+ic^{'}(\alpha)\gamma t}\partial_{\alpha}^{6}z_{\mu}(\alpha)+\\
&\lambda(\alpha)\int_{-\pi}^{\pi} \frac{\sin(z_{1}(\alpha)-z_{1}(\beta)+\tilde{f_1}(\alpha)-\tilde{f_1}(\beta))(\frac{\partial_{\alpha}^{6}z_{\mu}(\alpha)}{1+ic'(\alpha)\gamma t}-\frac{\partial_{\beta}^{6}z_{\mu}(\beta)}{1+ic'(\beta)\gamma t})(1+ic'(\beta)\gamma t)}{\cosh(z_{2}(\alpha)-z_{2}(\beta)+\tilde{f_1}(\alpha)-\tilde{f_1}(\beta))-\cos(z_{1}(\alpha)-z_{1}(\beta)+\tilde{f_1}(\alpha)-\tilde{f_1}(\beta))}d\beta\\
&+\lambda(\alpha)\int_{-\pi}^{\pi} \frac{\sin(z_{1}(\alpha)-z_{1}(\beta)+\tilde{f_1}(\alpha)-\tilde{f_1}(\beta))((\partial_{\alpha}^6\tilde{f_{\mu}})(\alpha)-(\partial_{\beta}^6\tilde{f_{\mu}})(\beta))(1+ic'(\beta)\gamma t)}{\cosh(z_{2}(\alpha)-z_{2}(\beta)+\tilde{f_1}(\alpha)-\tilde{f_1}(\beta))-\cos(z_{1}(\alpha)-z_{1}(\beta)+\tilde{f_1}(\alpha)-\tilde{f_1}(\beta))}d\beta\\
&+\sum_i O^i\\
&=T_1(z)+T_2(z)+T_3(z)+\sum_i O^i.
\end{align*}
Here $O^i$ terms contain at most 5th derivative on both $z$ and $\tilde{f}$.

Before we show the explicit form of $O^i$, we introduce some notations. 
Let 
\begin{align}\label{notationg}
V^{k}_g(\alpha) =(g_1(\alpha,t), g_2(\alpha,t), \partial_{\alpha}g_1(\alpha,t),\partial_{\alpha}g_2(\alpha,t), ...,\partial_{\alpha}^{k}g_1(\alpha,t),\partial_{\alpha}^{k}g_2(\alpha,t)),
\end{align}
\begin{align*}
\tilde{V}^k_{g}(\alpha)=(\frac{g_1(\alpha,t)}{1+ic'(\alpha)\gamma t},\frac{g_2(\alpha,t)}{1+ic'(\alpha)\gamma t},...,\partial_{\alpha}^k g_1(\alpha,t)\partial_{\alpha}^{5}(\frac{1}{1+ic'(\alpha)\gamma t}), \partial_{\alpha}^k g_2(\alpha,t)\partial_{\alpha}^{5}(\frac{1}{1+ic'(\alpha)\gamma t})).
\end{align*}
$V_{g,i}^k(\alpha)$ is the ith component in $V_{g}^k(\alpha)$ and  $\tilde{V}_{g,i}^k(\alpha)$ the ith component in $\tilde{V}_{g}^k(\alpha)$.

When we write $X_i(\alpha,t)$, we mean 
\begin{equation}\label{Xifunction}
X_i(\alpha,t)=\partial_{\alpha}^{l_i}(\frac{1}{1+ic'(\alpha)\gamma t}),
\end{equation}
with $l_i\leq 5$. 

A function $K_{-\sigma}^{j}(A,B)$, $K_{-\sigma}^{j}(A,B,C)$ is of $-\sigma$ type if for $A$, $B$, $C$ in $R^n$, it has the form
\begin{align}\label{k-sigma01}
&c_j\frac{\sin(A_1+B_1)^{m_1}\cos(A_1+B_1)^{m_2}}{(\cosh(A_2+B_2)-\cos(A_1+B_1))^{m_0}}\\\nonumber
    &\times(\sinh(A_2+B_2))^{m_3}(\cosh(A_2+B_2))^{m_4}\Pi_{j=1}^{m_5}(A_{\lambda_{j}})\Pi_{j=1}^{m_6}(B_{\lambda_{j,2}})\Pi_{j=1}^{m_7}(C_{\lambda_{j,3}}),
\end{align}
with $m_1+m_3+m_5+m_6+m_7-2m_0\geq -\sigma$. $c_j$ is a constant.

We claim that we can write $O^{i}$ as following three types, by separating the highest order term in the derivative. Here we omit the dependency on $\gamma$ and $t$.
\begin{align*}
O^{1,i}=\partial_{\alpha}^{b_i}\lambda(\alpha)\int_{-\pi}^{\pi}K_{-1}^{i}(V_z^{3}(\alpha)-V_z^{3}(\beta),V_{\tilde{f}}^{3}(\alpha)-V^{3}_{\tilde{f}}(\beta),\tilde{V}_{z}^3(\alpha)-\tilde{V}_{z}^3(\beta))X_{i'}(\beta)(\tilde{z}^{3}(\alpha)-\tilde{z}^3(\beta))d\beta,
\end{align*}
where $\tilde{z}^3\in V_z^{3} \cup  \tilde{V}_z^{3}\cup V_{\tilde{f}}^3$, $1\leq b_i\leq 5$.
 
\begin{align*}
O^{2,i}=\partial_{\alpha}^{b_i}\lambda(\alpha)\int_{-\pi}^{\pi}K_{-1}^{i}(V_z^{2}(\alpha)-V_z^{2}(\beta),V_{\tilde{f}}^{2}(\alpha)-V^{2}_{\tilde{f}}(\beta),\tilde{V}_{z}^2(\alpha)-\tilde{V}_{z}^2(\beta))X_{i'}(\beta)(\tilde{z}^{5}(\alpha)-\tilde{z}^5(\beta))d\beta,
\end{align*}
where $\tilde{z}^{5}\in V_z^{5} \cup  \tilde{V}_z^{5}\cup V_{\tilde{f}}^5$, $1\leq b_i\leq 5$. 

\begin{align*}
O^{3,i}=\partial_{\alpha}^{b_i}(\frac{ic(\alpha)\gamma}{1+ic^{'}(\alpha)\gamma t})\partial_{\alpha}^{b_i'}z_{\mu}(\alpha),
\end{align*}
where $1\leq b_i, b_i'\leq 5$.

Then we have the following lemmas
\begin{lemma}\label{01Oi}
We have
\[
\|O^i\|_{L^2[-\pi,\pi]}\lesssim C(\|z\|_{U}, \|z\|_{Arc}).
\]
\end{lemma}
\begin{proof}
Since $K_{-1}^{i}$ is of $-1$ type, we have
\[
\|K_{-1}^{i}(V_z^{3}(\alpha)-V_z^{3}(\beta),V_{\tilde{f}}^{3}(\alpha)-V^{3}_{\tilde{f}}(\beta),\tilde{V}_{z}^3(\alpha)-\tilde{V}_{z}^3(\beta))(\alpha-\beta)\|_{C^0[-2\delta,2\delta]\times[-\pi,\pi]}\lesssim C(\|z\|_{U},\|z\|_{Arc}),
\]
we could use lemma \ref{goodterm3} to get the result for $O^{1,i}$.
Moreover, we have
\[
\|K_{-1}^{i}(V_z^{2}(\alpha)-V_z^{2}(\beta),V_{\tilde{f}}^{2}(\alpha)-V^{2}_{\tilde{f}}(\beta),\tilde{V}_{z}^2(\alpha)-\tilde{V}_{z}^2(\beta))(\alpha-\beta)\|_{C^1[-2\delta,2\delta]\times[-\pi,\pi]}\lesssim C(\|z\|_{U},\|z\|_{Arc}),
\]
we then use lemma \ref{goodterm1} to get the estimate for $O^{2,i}$. $O^{3,i}$ can be bounded easily.
\end{proof}
\begin{lemma}\label{01tildefbound}
We have
\[
\|T_3(z)\|_{L^2[-\pi,\pi]}\lesssim C(\|z\|_{U}, \|z\|_{Arc}).
\]
\end{lemma}
\begin{proof}
Let 
\begin{align}\label{kequationnew}
K(z(\alpha)-z(\beta),\tilde{f}(\alpha)-\tilde{f}(\beta))=\frac{\sin(z_1(\alpha)-z_1(\beta)+\tilde{f}_1(\alpha)-\tilde{f}_1(\beta))}{\cosh(z_2(\alpha)-z_2(\beta)+\tilde{f}_2(\alpha)-\tilde{f}_2(\beta))-\cos(z_1(\alpha)-z_1(\beta)+\tilde{f}_1(\alpha)-\tilde{f}_1(\beta))}.
\end{align}
It is also of $K_{-1}$ type.
We have
\[
\|K(z(\alpha)-z(\beta),\tilde{f}(\alpha)-\tilde{f}(\beta))(\alpha-\beta)\|_{C^1[-2\delta,2\delta]\times[-\pi,\pi]}\lesssim C(\|z\|_{U},\|z\|_{Arc}).
\]
Then the result follows from lemma \ref{goodterm1}.
\end{proof}
Then we are left to deal with $T_1+T_2$.

By using the same notation as in lemma \ref{01tildefbound}, we have
\begin{align*}
   & T_2(z(\alpha))=\lambda(\alpha)p.v.\int_{-\pi}^{\pi} K(z(\alpha)-z(\beta),\tilde{f}(\alpha)-\tilde{f}(\beta))(1+ic'(\beta)\gamma t)d\beta\frac{\partial_{\alpha}^{6}z_{\mu}(\alpha)}{(1+ic'(\alpha)\gamma t)}\\
    &-\lambda(\alpha)p.v.\int_{-\pi}^{\pi} K(z(\alpha)-z(\beta),\tilde{f}(\alpha)-\tilde{f}(\beta))\partial_{\beta}^{6}z_{\mu}(\beta)d\beta\\
    &=T_{2,1}(z)(\alpha)+T_{2,2}(z)(\alpha).
\end{align*}
Moreover, we could further split the $T_{2,2}$ and have
\begin{align*}
    &T_{2,2}(z)(\alpha)= -\lambda(\alpha)p.v.\int_{-\pi}^{\pi} K(z(\alpha)-z(\beta),\tilde{f}(\alpha)-\tilde{f}(\beta))\partial_{\beta}^{6}z_{\mu}(\beta)d\beta\\
    &=-\lambda(\alpha)\lim_{\beta\to\alpha}(K(z(\alpha)-z(\beta),\tilde{f}(\alpha)-\tilde{f}(\beta))\tan(\frac{\alpha-\beta}{2}))p.v.\int_{-\pi}^{\pi}\cot(\frac{\alpha-\beta}{2})\partial_{\beta}^{6}z(\beta)d\beta\\
    &-\lambda(\alpha)\int_{-\pi}^{\pi}(K(z(\alpha)-z(\beta),\tilde{f}(\alpha)-\tilde{f}(\beta))\tan(\frac{\alpha-\beta}{2})-\lim_{\beta\to\alpha}(K(z(\alpha)-z(\beta),\tilde{f}(\alpha)-\tilde{f}(\beta))\tan(\frac{\alpha-\beta}{2})))\\
    &\cot(\frac{\alpha-\beta}{2})\partial_{\beta}^{6}z(\beta)d\beta\\
    &=T_{2,2,1}(z)+T_{2,2,2}(z).
\end{align*}
Since $K(z(\alpha)-z(\beta),\tilde{f}(\alpha)-\tilde{f}(\beta))$ is of $-1$ type, we have, 
\begin{align}\label{kcondition0}
  \|K(z(\alpha)-z(\beta),\tilde{f}(\alpha)-\tilde{f}(\beta))\tan(\frac{\alpha-\beta}{2})\|_{C^3_{\alpha,\beta}([-2\delta,2\delta]\times[-\pi,\pi])}\lesssim C(\|z\|_{Arc}+\|z\|_{U}).  
\end{align}
Let 
\begin{align*}
&\tilde{K}(z(\alpha)-z(\beta),\tilde{f}(\alpha)-\tilde{f}(\beta))\\
&=\frac{K(z(\alpha)-z(\beta),\tilde{f}(\alpha)-\tilde{f}(\beta))\tan(\frac{\alpha-\beta}{2})-\lim_{\beta\to\alpha}(K(z(\alpha)-z(\beta),\tilde{f}(\alpha)-\tilde{f}(\beta))\tan(\frac{\alpha-\beta}{2}))}{\tan(\frac{\alpha-\beta}{2})}.
\end{align*}
Then
\begin{align*}
  &\|\tilde{K}(z(\alpha)-z(\beta),\tilde{f}(\alpha)-\tilde{f}(\beta))\|_{C^2_{\alpha,\beta}([-2\delta,2\delta]\times[-\pi,\pi])}\\
  &\lesssim C(\|z\|_{Arc}+\|z\|_{U}). 
\end{align*}
We can do the integration by parts in the $T_{2,2,2}(z)$ and have
\begin{align*}
    &T_{2,2,2}(z)(\alpha)=-\lambda(\alpha)\int_{-\pi}^{\pi}\tilde{K}(z(\alpha)-z(\beta),\tilde{f}(\alpha)-\tilde{f}(\beta))\partial_{\beta}^{6}z(\beta)d\beta\\
    &=\lambda(\alpha)\int_{-\pi}^{\pi}\partial_{\beta}\tilde{K}(z(\alpha)-z(\beta),\tilde{f}(\alpha)-\tilde{f}(\beta))\partial_{\beta}^{5}z(\beta)d\beta.
\end{align*}
Therefore we have
\begin{align*}
    \|T_{2,2,2}(z)(\alpha)\|_{L^2[-\pi,\pi]}\lesssim C(\|z\|_{Arc}+\|z\|_{U})\|\partial_{\alpha}^{5}z(\alpha)\|_{L^2[-\pi,\pi]}.
\end{align*}
In conclusion, we have
\begin{align}\label{equationd5}
    &\frac{d\partial_{\alpha}^{5}z_{\mu}(\alpha,\gamma,t)}{dt}=T_1(z)+T_{2,1}(z)+T_{2,2,1}(z)+(T_{2,2,2}(z)+T_3(z)+\sum_{i}O^{i}(z))\\\nonumber
    &=(\frac{ic(\alpha)\gamma}{1+ic^{'}(\alpha)\gamma t}+\frac{\lambda(\alpha)}{1+ic'(\alpha)\gamma t}p.v.\int_{-\pi}^{\pi} K(z(\alpha)-z(\beta),\tilde{f}(\alpha)-\tilde{f}(\beta))(1+ic'(\beta)\gamma t)d\beta)\partial_{\alpha}^{6}z_{\mu}(\alpha)\\\nonumber
    &-\lambda(\alpha)\lim_{\beta\to\alpha}(K(z(\alpha)-z(\beta),\tilde{f}(\alpha)-\tilde{f}(\beta))\tan(\frac{\alpha-\beta}{2}))2\pi \Lambda(\partial_{\alpha}^5z)(\alpha)\\\nonumber
    &+(T_{2,2,2}(z)+T_3(z)+\sum_{i}O^{i}(z)),
\end{align}
where $\Lambda$ is $(-\Delta)^{\frac{1}{2}}$  on the Torus $\mathbb{T}$ of length $2\pi$ and  
\begin{align}\label{equationd5condition}
\|(T_{2,2,2}(z)+T_3(z)+\sum_{i}O^{i})\|_{L^2[-\pi,\pi]}\lesssim C(\|z\|_{Arc}+\|z\|_{U}).
\end{align}
Then we have
\begin{align}\label{energy main term}
    &\frac{d}{dt}\|\partial_{\alpha}^5z(\alpha,\gamma,t)\|_{L^2_{\alpha}[-\pi,\pi]}^2=2\Re\int_{-\pi}^{\pi}\partial_{\alpha}^5z(\alpha,\gamma,t)\cdot\overline{\partial_{\alpha}^5\frac{d}{dt}z(\alpha,\gamma,t)}d\alpha\\\nonumber
    &=\sum_{\mu=1,2}(\underbrace{2\Re\int_{-\pi}^{\pi}\partial_{\alpha}^5z_{\mu}(\alpha)\overline{(\frac{ic(\alpha)\gamma}{1+ic^{'}(\alpha)\gamma t}+\frac{\lambda(\alpha)}{1+ic'(\alpha)\gamma t}p.v.\int_{-\pi}^{\pi} K(z(\alpha)-z(\beta),\tilde{f}(\alpha)-\tilde{f}(\beta))(1+ic'(\beta)\gamma t)d\beta)}}_{\text{main term}}\\\nonumber
    &\underbrace{\overline{\cdot\partial_{\alpha}^{6}z_{\mu}(\alpha)}d\alpha}_{\text{main term}}\\\nonumber
    &\underbrace{-2\Re\int_{-\pi}^{\pi}\partial_{\alpha}^5z_{\mu}(\alpha)\overline{\lambda(\alpha)\lim_{\beta\to\alpha}(K(z(\alpha)-z(\beta),\tilde{f}(\alpha)-\tilde{f}(\beta))\tan(\frac{\alpha-\beta}{2}))2\pi \Lambda(\partial_{\alpha}^5z_{\mu})(\alpha)}d\alpha}_{\text{main term}})\\\nonumber
    &+B.T,
\end{align}
where $B.T.\leq C(\|z\|_{Arc}+\|z\|_{U}).$
Next we show the following lemma to control the main terms.
\begin{lemma}\label{garding apply}
If $L_1(\alpha)$, $L_2(\alpha)\in C^2(\mathbb{T})$, $-\Re L_1(\alpha)\geq |\Im L_2(\alpha)|$, $h\in H^1(\mathbb{T})$, then we have 
\begin{align*}
   \Re( \int_{-\pi}^{\pi}h(\alpha)\overline{L_1(\alpha)(\Lambda h)(\alpha)}d\alpha+\int_{-\pi}^{\pi}h(\alpha)\overline{L_2(\alpha)\partial_{\alpha}h(\alpha)}d\alpha)\leq C(\|L_1\|_{C^2(\mathbb{T})}+\|L_2\|_{C^2(\mathbb{T})})\|h(\alpha)\|_{L^2}^{2}.
    \end{align*}
\end{lemma}
\begin{proof}
First, we have
\begin{align*}
    &I_2=\Re(\int_{-\pi}^{\pi}h(\alpha)\overline{L_2(\alpha)\partial_{\alpha}h(\alpha)}d\alpha)\\
    &=\int_{-\pi}^{\pi}\Re L_2(\alpha)(\Re h(\alpha)\Re h'(\alpha)+\Im h(\alpha)\Im h'(\alpha))d\alpha\\
    &-\int_{-\pi}^{\pi}\Im L_2(\alpha)(\Re h(\alpha)\Im h'(\alpha)-\Im h(\alpha)\Re h'(\alpha))d\alpha\\
    &=I_{2,1}+I_{2,s}.
\end{align*}
We could do the integration by parts to $I_{2,1}$ and have
\begin{align*}
    |I_{2,1}|=|\frac{1}{2}\int_{-\pi}^{\pi}\frac{d\Re L_2(\alpha)}{d\alpha}((\Re h(\alpha))^2+(\Im h(\alpha))^2)d\alpha|\lesssim C(\|L_2(\alpha)\|_{C^1(\mathbb{T})})\|h(\alpha)\|_{L^2}^2.
\end{align*}
Moreover, 
\begin{align*}
    &I_1=\Re(\int_{-\pi}^{\pi}h(\alpha)\overline{L_1(\alpha)(\Lambda h)(\alpha)}d\alpha)\\
    &=\int_{-\pi}^{\pi}\Re L_1(\alpha)(\Re h(\alpha)\Lambda \Re h(\alpha)+\Im h(\alpha)\Lambda \Im h(\alpha))d\alpha\\
    &-\int_{-\pi}^{\pi}\Im L_1(\alpha)(\Re h(\alpha)\Lambda \Im h(\alpha)-\Im h(\alpha)\Lambda \Re h(\alpha))d\alpha\\
    &=I_{1,M}+I_{1,2}.
\end{align*}
We can still do the integration by parts to the $I_{1,2}$ and have 
\begin{align*}
    &\int_{-\pi}^{\pi}\Im L_1(\alpha)(\Re h(\alpha)\Lambda \Im h(\alpha)-\Im h(\alpha)\Lambda \Re h(\alpha))d\alpha\\
    &=\int_{-\pi}^{\pi}(\Lambda(\Im L_1\Re h)(\alpha)-\Im L_1(\alpha)\Lambda \Re h(\alpha))\Im h(\alpha)d\alpha.
\end{align*}
Moreover, for any $g_1\in H^2(\mathbb{T})$, $g_2\in H^1(\mathbb{T})$, we have
\begin{align*}
    \|\Lambda(g_1g_2)-g_1\Lambda g_2\|_{L^2(\mathbb{T})}\lesssim\|g_2\|_{L^2(\mathbb{T})}\|g_1\|_{H^2(\mathbb{T})}.
\end{align*}
Hence 
\begin{align*}
&I_{1,2}\lesssim \|\Im L_1(\alpha)\|_{H^2(\mathbb{T})}\|\Im h(\alpha)\|_{L^2(\mathbb{T})}\|\Re h(\alpha)\|_{L^2(\mathbb{T})}\\
&\lesssim  \|\Im L_1(\alpha)\|_{H^2(\mathbb{T})}\| h(\alpha)\|_{L^2(\mathbb{T})}^2.
\end{align*}
Now we are left to control $I_{1,M}+I_{2,s}$. We have
\begin{align*}
    &I_{1,M}+I_{2,s}=
    \int_{-\pi}^{\pi}\Re L_1(\alpha)(\Re h(\alpha)\Lambda \Re h(\alpha)+\Im h(\alpha)\Lambda \Im h(\alpha))d\alpha\\
    &+\int_{-\pi}^{\pi}\Im L_2(\alpha)(\Re h(\alpha)\Im h'(\alpha)-\Im h(\alpha)\Re h'(\alpha))d\alpha.
\end{align*}
Now we use a lemma from \cite[Section 2.4]{Castro-Cordoba-Fefferman-Gancede}.
\begin{lemma}\label{Garding}
Let $a$, $b$ be real valued functions on $\mathbb{T}$, $a(\alpha)\geq|b(\alpha)|$ and satisfying $a,b \in C^2(\mathbb{T})$. Then we have
\[
\Re\int_{\mathbb{T}}\overline{f(x)}(a(x)\Lambda f(x)+b(x)if'(x))dx\geq -C(\|a\|_{C^2(\mathbb{T})}+\|b\|_{C^2(\mathbb{T})})\int_{\mathbb{T}}|f(x)|^2dx.
\]
\end{lemma}
Then from lemma \ref{Garding}, we have
\[
I_{1,M}+I_{2,s}\lesssim C(\|L_1\|_{C^2(\mathbb{T})}+\|L_2\|_{C^2(\mathbb{T})})\|h(\alpha)\|_{L^2}^{2}.
\]
Then we get the result.

\end{proof}
Now let 
\begin{align}\label{Lz1}
&L_{z}^1(\alpha,\gamma,t)=-2\pi\lim_{\beta\to 0}(K(z(\alpha,\gamma,t)-z(\beta,\gamma,t),\tilde{f}(\alpha,t)-\tilde{f}(\beta,t))\tan(\frac{\alpha-\beta}{2}))\\\nonumber
&=-2\pi\frac{\partial_{\alpha}z_1(\alpha,\gamma,t)+\partial_{\alpha}\tilde{f}_1(\alpha,t)}{(\partial_{\alpha}z_1(\alpha,\gamma,t)+\partial_{\alpha}\tilde{f}_1(\alpha,t))^2+(\partial_{\alpha}z_2(\alpha,\gamma,t)+\partial_{\alpha}\tilde{f}_2(\alpha,t))^2},
 \end{align}
and
\begin{align}\label{Lz2}
L_{z}^2(\alpha,\gamma,t)=(\frac{ic(\alpha)\gamma}{1+ic^{'}(\alpha)\gamma t}+\frac{1}{1+ic'(\alpha)\gamma t}p.v.\int_{-\pi}^{\pi} K(z(\alpha,\gamma,t)-z(\beta,\gamma,t),\tilde{f}(\alpha,t)-\tilde{f}(\beta,t))(1+ic'(\beta)\gamma t)d\beta).
\end{align}
Since $\supp c(\alpha)\subset \{\alpha| \lambda(\alpha)=1\}$, from \eqref{energy main term}, we have
\begin{align*}
    &\frac{d}{dt}\|\partial_{\alpha}^5z(\alpha,\gamma,t)\|_{L^2_{\alpha}[-\pi,\pi]}^2\\
    &=\sum_{\mu=1,2}\underbrace{2\Re\int_{-\pi}^{\pi}\partial_{\alpha}^5z_{\mu}(\alpha)\overline{\lambda(\alpha)L_{z}^2(\alpha)\partial_{\alpha}^{6}z_{\mu}(\alpha)}d\alpha}_{\text{main term}}\underbrace{+2\Re\int_{-\pi}^{\pi}\partial_{\alpha}^5z_{\mu}(\alpha)\overline{\lambda(\alpha)L_{z}^1(\alpha) \Lambda(\partial_{\alpha}^5z_{\mu})(\alpha)}d\alpha}_{\text{main term}}\\\nonumber
    &+B.T.
\end{align*}

From lemma \eqref{garding apply}, if $-\Re L_{z}^1(\alpha)\geq |\Im L_z^2(\alpha)|$ for $\alpha\in[-2\delta,2\delta]$,  since $\supp \lambda\subset[-2\delta,2\delta]$, we have
\begin{align*}
     &\frac{d}{dt}\|\partial_{\alpha}^5z(\alpha,\gamma,t)\|_{L^2_{\alpha}[-\pi,\pi]}^2\lesssim B.T.
\end{align*}
Moreover, when $t=0$, $\alpha\in[-2\delta,2\delta]$, from \eqref{Lz1}, \eqref{Lz2} and \eqref{kequationnew} we have
\begin{align*}
&-\Re L_z^1(\alpha,\gamma,0)=2\pi \frac{\partial_{\alpha}f_1^c(\alpha,0)+\partial_{\alpha}\tilde{f}_1(\alpha,0)}{(\partial_{\alpha}f_1^c(\alpha,0)+\partial_{\alpha}\tilde{f}_1(\alpha,0))^2+(\partial_{\alpha}f_2^c(\alpha,0)+\partial_{\alpha}\tilde{f}_2(\alpha,0))^2}\\
&=2\pi \frac{\partial_{\alpha}f_1(\alpha,0)}{(\partial_{\alpha}f_1(\alpha,0))^2+(\partial_{\alpha}f_2(\alpha,0))^2},
\end{align*}
and
\begin{align*}
&\Im L_z^2(\alpha,0)=ic(\alpha)\gamma+\Im (p.v.\int_{-\pi}^{\pi} K(f^c(\alpha,0)-f^c(\beta,0),\tilde{f}(\alpha,0)-\tilde{f}(\beta,0))d\beta)\\
&=ic(\alpha)\gamma.
\end{align*}
From \eqref{RTlocal}, we could choose $\delta_c$ in \eqref{deltac} to be sufficiently small and have
\begin{align*}
    \inf_{\alpha\in[-2\delta,2\delta]}(-\Re L_z^1(\alpha,\gamma,0)-|\Im L_z^2(\alpha,\gamma,0)|)>0.
\end{align*}
Then let 
\begin{align*}
    \|z\|_{RT}(t)=\sup_{\alpha\in[-2\delta,2\delta]}\frac{1}{|\Re L_z^1(\alpha,\gamma,t)+|\Im L_z^2(\alpha,\gamma,t)||}.
\end{align*}
If $\|z\|_{RT}(t)< \infty$, we have 
\begin{align}\label{energyh51}
\frac{d}{dt}\|z\|_{U}^2\lesssim C(\|z\|_{U}+\|z\|_{Arc}).
\end{align}

Therefore, we could let $\|z\|_{\tilde{U}}=\|z\|_{U}+\|z\|_{Arc}+\|z\|_{RT}$.  From \eqref{energyh51}, and the following lemma \ref{RTandArc}, we have 
\begin{align*}
   \frac{d}{dt}\|z\|_{\tilde{U}}^2\lesssim C(\|z\|_{\tilde{U}}).
\end{align*}
Then $\|z(\alpha,\gamma,t)\|_{\tilde{U}}$ is bounded for
sufficiently small time $t_1$. We claim that the bound and the time can be chosen such that it holds for all $\gamma\in[-1,1]$.
\begin{lemma}\label{RTandArc}
We have the following two estimates:
\begin{align*}
    \frac{d}{dt}\|z\|_{Arc}\lesssim C(\|z\|_{\tilde{U}}),
\end{align*}
and
\begin{align*}
    \frac{d}{dt}\|z\|_{RT}\lesssim C(\|z\|_{\tilde{U}}).
\end{align*}
\end{lemma}
\begin{proof}
For the $\|z\|_{RT}$, we have
\begin{align*}
    &\frac{d}{dt}\sup_{\alpha\in[-2\delta,2\delta]}\frac{1}{|\Re L_z^1(\alpha)+|\Im L_z^2(\alpha)||}\leq \sup_{\alpha\in[-2\delta,2\delta]}\frac{1}{|\Re L_z^1(\alpha)+|\Im L_z^2(\alpha)||^2}(\|\frac{d}{dt}L_z^1(\alpha)\|_{L^{\infty}_{\alpha}[-2\delta,2\delta]}+\|\frac{d}{dt} L_z^2(\alpha)\|_{L^{\infty}_{\alpha}[-2\delta,2\delta]}).
\end{align*}
From \eqref{Lz1}, we have 
\begin{align*}
    &\|\frac{d}{dt}L_z^1(\alpha)\|_{L^{\infty}_{\alpha}[-2\delta,2\delta]}\lesssim C(\|z\|_{\tilde{U}})(\|\frac{d}{dt}z(\alpha,\gamma,t)\|_{C^1(\mathbb{T})}+\|\frac{d}{dt}\tilde{f}(\alpha,t)\|_{C^1(\mathbb{T})})\\
    &\lesssim C(\|z\|_{\tilde{U}})(\|T(z)\|_{C^1(\mathbb{T})}+C).
\end{align*}
From \eqref{zequationsimplify}, it is easy to get
\begin{align*}
    \|T(z)\|_{C^1(\mathbb{T})}\lesssim C(\|z\|_{\tilde{U}}).
\end{align*}
Then \begin{align*}
    &\|\frac{d}{dt}L_z^1(\alpha)\|_{L^{\infty}_{\alpha}[-2\delta,2\delta]}\lesssim C(\|z\|_{\tilde{U}}).
\end{align*}
From \eqref{Lz2}, we have
\begin{align*}
    &\|\frac{d}{dt}L_z^2(\alpha,\gamma,t)\|_{L^{\infty}_{\alpha}[-2\delta,2\delta]}\lesssim C+ C\|\frac{d}{dt}p.v.\int_{-\pi}^{\pi} K(z(\alpha)-z(\beta),\tilde{f}(\alpha)-\tilde{f}(\beta))(1+ic'(\beta)\gamma t)d\beta\|_{L^{\infty}_{\alpha}[-2\delta,2\delta]}\\
    &+\|p.v.\int_{-\pi}^{\pi} K(z(\alpha)-z(\beta),\tilde{f}(\alpha)-\tilde{f}(\beta))(1+ic'(\beta)\gamma t)d\beta\|_{L^{\infty}_{\alpha}[-2\delta,2\delta]}\\
    &\leq \|p.v.\int_{-\pi}^{\pi} \nabla_1 K(z(\alpha)-z(\beta),\tilde{f}(\alpha)-\tilde{f}(\beta))\cdot(\frac{dz}{dt}(\alpha)-\frac{dz}{dt}(\beta)+\frac{d\tilde{f}}{dt}(\alpha)-\frac{d\tilde{f}}{dt}(\beta))(1+ic'(\beta)\gamma t)d\beta\|_{L^{\infty}_{\alpha}[-2\delta,2\delta]}\\
    &+\|p.v.\int_{-\pi}^{\pi} K(z(\alpha)-z(\beta),\tilde{f}(\alpha)-\tilde{f}(\beta))(ic'(\beta)\gamma )d\beta\|_{L^{\infty}_{\alpha}[-2\delta,2\delta]}\\
    &+\|p.v.\int_{-\pi}^{\pi} K(z(\alpha)-z(\beta),\tilde{f}(\alpha)-\tilde{f}(\beta))(1+ic'(\beta)\gamma t)d\beta\|_{L^{\infty}_{\alpha}[-2\delta,2\delta]}+C\\
    &=Term_{2,1}+Term_{2,2}+Term_{2,3}+C.
\end{align*}
From condition \eqref{kcondition0} and lemma \ref{goodterm0}, we have
\[
Term_{2,2}+Term_{2,3}\lesssim C(\|z\|_{\tilde{U}}).
\]
Moreover, $\nabla_1 K(z(\alpha)-z(\beta),\tilde{f}(\alpha)-\tilde{f}(\beta))$ is of $-2$ type, and
\begin{align*}
    &\|\nabla_1 K(z(\alpha)-z(\beta),\tilde{f}(\alpha)-\tilde{f}(\beta))\cdot(\frac{dz}{dt}(\alpha)-\frac{dz}{dt}(\beta)+\frac{d\tilde{f}}{dt}(\alpha)-\frac{d\tilde{f}}{dt}(\beta))(\alpha-\beta)\|_{C^1([-2\delta,2\delta]\times[-\pi,\pi])}\\
    &\lesssim(\|T(z)\|_{C^2(\mathbb{T})}+C)C(\|z\|_{\tilde{U}}).
\end{align*}

From \eqref{zequationsimplify}, it is easy to get
\begin{align*}
    \|T(z)\|_{C^2(\mathbb{T})}\lesssim C(\|z\|_{\tilde{U}}).
\end{align*}
Then $Term_{2,1}\lesssim C(\|z\|_{\tilde{U}})$. Hence
\begin{align*}
    &\|\frac{d}{dt}L_z^2(\alpha)\|_{L^{\infty}_{\alpha}[-2\delta,2\delta]}\lesssim C(\|z\|_{\tilde{U}}).
\end{align*}
Then we have the estimate 
\begin{align*}
    &\frac{d}{dt}\|z\|_{RT}=\frac{d}{dt}\sup_{\alpha\in[-2\delta,2\delta]}\frac{1}{|\Re L_z^1(\alpha)+|\Im L_z^2(\alpha)||}\lesssim C(\|z\|_{\tilde{U}}).
\end{align*}
Moreover, we have
\begin{align*}
    &\frac{d}{dt}\|z\|_{Arc}=\frac{d}{dt}\sup_{\alpha\in[-2\delta,2\delta],\beta\in[-\pi,\pi]}\big|\frac{1}{(\frac{\cosh(z_{2}(\alpha)-z_{2}(\beta)+\tilde{f_1}(\alpha)-\tilde{f_1}(\beta))-\cos(z_{1}(\alpha,\gamma,t)-z_{1}(\beta,\gamma,t)+\tilde{f_1}(\alpha)-\tilde{f_1}(\beta))}{(\alpha-\beta)^2})}\big|\\
    &\leq \sup_{\alpha\in[-2\delta,2\delta],\beta\in[-\pi,\pi]}\big|\frac{1}{(\frac{\cosh(z_{2}(\alpha)-z_{2}(\beta)+\tilde{f_1}(\alpha)-\tilde{f_1}(\beta))-\cos(z_{1}(\alpha,\gamma,t)-z_{1}(\beta,\gamma,t)+\tilde{f_1}(\alpha)-\tilde{f_1}(\beta))}{(\alpha-\beta)^2})}\big|^2\\
    &\cdot\|\frac{d}{dt}(\frac{\cosh(z_{2}(\alpha)-z_{2}(\beta)+\tilde{f_1}(\alpha)-\tilde{f_1}(\beta))-\cos(z_{1}(\alpha,\gamma,t)-z_{1}(\beta,\gamma,t)+\tilde{f_1}(\alpha)-\tilde{f_1}(\beta))}{(\alpha-\beta)^2})\|_{L^{\infty}_{\alpha,\beta}[-\pi,\pi]\times[-\pi,\pi]}\\
    &\leq \|z\|_{Arc}^2\\
    &(\|\frac{\sinh(z_{2}(\alpha)-z_{2}(\beta)+\tilde{f_2}(\alpha)-\tilde{f_2}(\beta))}{\alpha-\beta}\frac{\frac{d}{dt}(z_2(\alpha)-z_2(\beta))+\frac{d}{dt}(\tilde{f}_2(\alpha)-\tilde{f}_2(\beta))}{(\alpha-\beta)}\|_{L^{\infty}_{\alpha,\beta}[-\pi,\pi]\times[-\pi,\pi]}\\
    &+\|\frac{\sin(z_{1}(\alpha)-z_{1}(\beta)+\tilde{f_1}(\alpha)-\tilde{f_1}(\beta))}{\alpha-\beta}\frac{\frac{d}{dt}(z_1(\alpha)-z_1(\beta))+\frac{d}{dt}(\tilde{f}_1(\alpha)-\tilde{f}_2(\beta))}{(\alpha-\beta)}\|_{L^{\infty}_{\alpha,\beta}[-\pi,\pi]\times[-\pi,\pi]})\\
    &\lesssim  \|z\|_{Arc}^2\|z\|_{U}(\|T(z)\|_{C^{1}_{\alpha}[-\pi,\pi]}+C)\\
    &\lesssim C(\|z\|_{\tilde{U}}).
\end{align*}
We also introduce a corollary here to be used in the later section.
\begin{corollary}\label{maintermcorollary}
For $g(\alpha)\in H^1(\mathbb{T})$, if $z\in H^5(\mathbb{T})$, $\|z\|_{Arc}< \infty$ and $-\Re L_z^1(\alpha)-|\Im L_z^2(\alpha)|>0$ when $\alpha\in [-2\delta,2\delta]$, $\gamma \in[-1,1]$, then we have
\begin{align*}
    &\sup_{\gamma \in[-1,1]}(\int_{-\pi}^{\pi}g(\alpha)\overline{\lambda(\alpha)\int_{-\pi}^{\pi} K(z(\alpha)-z(\beta),\tilde{f}(\alpha)-\tilde{f}(\beta))(\frac{\partial_{\alpha}g(\alpha)}{1+ic'(\alpha)\gamma t}-\frac{\partial_{\alpha}g(\beta)}{1+ic'(\beta)\gamma t})d\beta}d\alpha\\
    &+\int_{-\pi}^{\pi}g(\alpha)\overline{\partial_{\alpha}g(\alpha)\frac{ic(\alpha)\gamma}{1+ic'(\alpha)\gamma t}}d\alpha)\lesssim \|g\|_{L^2}^2.
\end{align*}
\end{corollary}
\end{proof}
\subsection{Approximation for the Picard theorem}

 Now we approximate the problem and have the following equations,

\begin{align*}
&\frac{d z^n(\alpha, \gamma,t)}{dt}=\varphi_n*T^n(\varphi_n*z^n)=\varphi_n*\bigg(\frac{ic(\alpha)\gamma}{1+ic^{'}(\alpha)\gamma t}\partial_{\alpha}(\varphi_n*z^n)(\alpha)\bigg)+\\
&\varphi_n*\bigg{(}\lambda(\alpha)\int_{-\pi}^{\pi} K^n((\varphi_n*z^n)(\alpha)-(\varphi_n*z^n)(\beta),\tilde{f}(\alpha)-\tilde{f}(\beta))(\frac{\partial_{\alpha}(\varphi_n*z_{\mu}^n)(\alpha)}{1+ic'(\alpha)\gamma t}-\frac{\partial_{\beta}(\varphi_n*z_{\mu}^n)(\beta)}{1+ic'(\beta)\gamma t})(1+ic'(\beta)\gamma t)d\beta\bigg{)}\\
&+\varphi_n*\bigg{(}\lambda(\alpha)\int_{-\pi}^{\pi} K^n((\varphi_n*z^n)(\alpha)-(\varphi_n*z^n)(\beta),\tilde{f}(\alpha)-\tilde{f}(\beta))(\tilde{f}_{\mu}(\alpha)-\tilde{f}_{\mu}(\beta))(1+ic'(\beta)\gamma t)d\beta\bigg{)},
\end{align*}
where
\begin{align*}
&K^{n}((\varphi_n*z^n)(\alpha)-(\varphi_n*z^n)(\beta),\tilde{f}(\alpha)-\tilde{f}(\beta))\\
&=\sin(\Delta((\varphi_n*z_{1}^n)(\alpha)+\tilde{f_1}(\alpha)))\\
&\cdot\frac{1}{|\cosh(\Delta((\varphi_n*z_{2}^n)(\alpha)+\tilde{f_2}(\alpha)))-\cos(\Delta((\varphi_n*z_{1}^n)(\alpha)+\tilde{f_1}(\alpha)))|^2+\frac{1}{n}\sin(\frac{\alpha-\beta}{2})^2}\\
&\cdot\overline{\cosh(\Delta((\varphi_n*z_{2}^n)(\alpha)+\tilde{f_2}(\alpha)))-\cos(\Delta((\varphi_n*z_{1}^n)(\alpha)+\tilde{f_1}(\alpha)))},
\end{align*}
with
\begin{equation}\label{Delta01}
    \Delta((\varphi_n*z_{2}^n)(\alpha)+\tilde{f_2}(\alpha))=(\varphi_n*z_{2}^n)(\alpha)-(\varphi_n*z_{2}^n)(\beta)+\tilde{f_2}(\alpha)-\tilde{f_2}(\beta),
\end{equation}
\begin{equation}\label{Delta02}
    \Delta((\varphi_n*z_{1}^n)(\alpha)+\tilde{f_1}(\alpha))=(\varphi_n*z_{1}^n)(\alpha)-(\varphi_n*z_{1}^n)(\beta)+\tilde{f_1}(\alpha)-\tilde{f_1}(\beta),
\end{equation}
and initial value $z^n(\alpha, \gamma, 0)=\varphi_{n}*f(\alpha,0)$. Here the convolution of $\varphi_n$ is the projection to the finite Fourier modes of $\alpha$.

By the Picard theorem, for any fixed $\gamma \in [-1,1]$, there exists solutions in $C^1([0,t_{n}], H^5_{\alpha}(\mathbb{T}))$.
Moreover, by the structure of our approximation, we have $z^n=\varphi_{n}*z^n$, and for $1\leq j\leq 5$,
\begin{align*}
    &\frac{d}{dt}\int_{-\pi}^{\pi}|\partial_{\alpha}^jz^n(\alpha,\gamma,t)|^2d\alpha\\
    &=2\Re\int_{-\pi}^{\pi}\partial_{\alpha}^jz^n(\alpha,\gamma,t)\overline{\partial_{\alpha}^j(\varphi_n*T^n(\varphi_n*{z^n}))}d\alpha\\
    &=2\Re\int_{-\pi}^{\pi}\partial_{\alpha}^j(\varphi_n*{z^n})(\alpha,\gamma,t)\overline{\partial_{\alpha}^j(T^n(\varphi_n*{z^n}))}d\alpha.
\end{align*}
Then we can do the similar energy estimate as the previous section by letting
\[
\|z^n\|_{\tilde{U}^n}=\|z^n\|_{H^5[-\pi,\pi]}+\|z^n\|_{RT^n}+\|z^n\|_{Arc^n},
\]
where 
\begin{align*}
&\|z^n\|_{Arc^n}=\sup_{\alpha\in[-2\delta,2\delta],\beta\in[-\pi,\pi]}\frac{|\cosh(\Delta((\varphi_n*z_{2}^n)(\alpha)+\tilde{f_2}(\alpha)))-\cos(\Delta((\varphi_n*z_{1}^n)(\alpha)+\tilde{f_1}(\alpha)))|(\alpha-\beta)^2}{|\cosh(\Delta((\varphi_n*z_{2}^n)(\alpha)+\tilde{f_2}(\alpha)))-\cos(\Delta((\varphi_n*z_{1}^n)(\alpha)+\tilde{f_1}(\alpha)))|^2+\frac{1}{n}\sin(\frac{\alpha-\beta}{2})^2},
\end{align*}
with  $\Delta((\varphi_n*z_{2}^n)(\alpha)+\tilde{f_2}(\alpha))$ and $\Delta((\varphi_n*z_{1}^n)(\alpha)+\tilde{f_1}(\alpha))$ from \eqref{Delta01}, \eqref{Delta02},
and letting
\begin{align*}
\|z^n\|_{RT^n}=\sup_{\alpha\in[-2\delta,2\delta]}\frac{1}{|\Re L_{z}^{1,n}(\alpha)+|\Im L_z^{2,n}(\alpha)||},
\end{align*}
with
\[
L_{z^n}^{1,n}(\alpha,\gamma,t)=-2\pi\lim_{\beta\to 0}(K^{n}((\varphi_n*z^n)(\alpha)-(\varphi_n*z^n)(\beta),\tilde{f}(\alpha)-\tilde{f}(\beta))\tan(\frac{\alpha-\beta}{2})),
\]
\[
L_{z^n}^{2,n}(\alpha,\gamma,t)=(\frac{ic(\alpha)\gamma}{1+ic^{'}(\alpha)\gamma t}+\frac{1}{1+ic'(\alpha)\gamma t}p.v.\int_{-\pi}^{\pi} K^n((\varphi_n*z^n)(\alpha)-(\varphi_n*z^n)(\beta),\tilde{f}(\alpha)-\tilde{f}(\beta))(1+ic'(\beta)\gamma t)d\beta).
\]

Then we can use the similar energy estimate and the compactness argument to show there exist a solution
\begin{equation}\label{zspace}
z(\alpha,\gamma,t)\in L^\infty([0,t_1],H^5_{\alpha}(\mathbb{T})),
\end{equation}
satisfying 
\begin{align}\label{integral equation}
z(\alpha,\gamma,t)=\int_{0}^{t}T(z)d\tau+f^c(\alpha,0),
\end{align}
for sufficiently small time $t_1$.
Moreover, 
\begin{align}\label{arcconditions}
\|z\|_{Arc}<\infty,
\end{align}
and
\begin{align}\label{RTconditions}
-\Re L_z^1(\alpha)-|\Im L_z^2(\alpha)|>0, \text{ when } \alpha\in[-2\delta,2\delta].
\end{align}
 Since the energy estimate has a bound for all $\gamma \in[-1,1]$, we have a existence time $t_1$ that holds for all $\gamma$.

Now we abuse the notation and write $T(z)$ as $T(z(\alpha,\gamma,t),\gamma,t)$. We have the following lemma:
\begin{lemma}\label{Tproperty}
    For any $g(\alpha),h(\alpha)\in H^{j+1}(\mathbb{T})$, $j=3,4,$ $\|g\|_{Arc}< \infty$ and $\|h\|_{Arc}<\infty$, we have
    \begin{align}\label{upperbound01}
    \|T(g(\alpha),\gamma,t)\|_{H^j(\mathbb{T})}\lesssim 1,
    \end{align}
     \begin{align}
    \|T(g(\alpha),\gamma,t)-T(h(\alpha),\gamma,t)\|_{H^j(\mathbb{T})}\lesssim \|g(\alpha)-g(\alpha)\|_{H^{j+1}(\mathbb{T})},
    \end{align}
     \begin{align}
    \|T(g(\alpha),\gamma,t)-T(g(\alpha),\gamma,t')\|_{H^j(\mathbb{T})}\lesssim |t-t'|.
    \end{align}
\end{lemma}
\begin{proof}
We only show \eqref{upperbound01} and the left can be shown in the same way.
From \eqref{zequationsimplify}, we have
\begin{align*}
    &T(g(\alpha),\gamma,t)=\frac{ic(\alpha)\gamma}{1+ic^{'}(\alpha)\gamma t}\partial_{\alpha}g(\alpha)+\\\nonumber
&\lambda(\alpha)\int_{-\pi}^{\pi} K(g(\alpha)-g(\beta),\tilde{f}(\alpha)-\tilde{f}(\beta))(\frac{\partial_{\alpha}g(\alpha)}{1+ic'(\alpha)\gamma t}-\frac{\partial_{\beta}g(\beta)}{1+ic'(\beta)\gamma t})(1+ic'(\beta)\gamma t)d\beta\\\nonumber
&+\lambda(\alpha)\int_{-\pi}^{\pi} K(g(\alpha)-g(\beta),\tilde{f}(\alpha)-\tilde{f}(\beta))(\partial_{\alpha}\tilde{f}(\alpha)-\partial_{\beta}\tilde{f}(\beta))(1+ic'(\beta)\gamma t)d\beta\\
&=T_{1}(g(\alpha),\gamma,t)+T_{2}(g(\alpha),\gamma,t)+T_{3}(g(\alpha),\gamma,t).
\end{align*}
It is trivial that $T_1$ satisfying the \eqref{upperbound01} since $c(\alpha)$ is sufficiently smooth. 

Moreover, the $\partial_{\alpha}\tilde{f}\in H^5(\mathbb{T})$ and is more regular than $\frac{\partial_{\alpha}g(\alpha)}{1+ic'(\alpha)\gamma t}.$ Hence we only consider $T_2$.
For $T_2$, we have
\begin{align}\label{kgcondition}
\|K(g(\alpha)-g(\beta),\tilde{f}(\alpha)-\tilde{f}(\beta))(\alpha-\beta)\|_{C^{j-2}([-2\delta,2\delta]\times\mathbb{T})}\lesssim 1.
\end{align}
Then
\begin{align*}
    \|T_2\|_{L^2(\mathbb{T})}\lesssim  \|T_2\|_{L^\infty(\mathbb{T})}\lesssim 1.
\end{align*}
Moreover, we can use the notation from \eqref{k-sigma01}, \eqref{notationg}, and get
\begin{align*}
    &\partial_{\alpha}^{j}T_2(g)=\\
    &\lambda(\alpha)\int_{-\pi}^{\pi} K(g(\alpha)-g(\beta),\tilde{f}(\alpha)-\tilde{f}(\beta))(\frac{\partial_{\alpha}^{j+1}g(\alpha)}{1+ic'(\alpha)\gamma t}-\frac{\partial_{\beta}^{j+1}g(\beta)}{1+ic'(\beta)\gamma t})(1+ic'(\beta)\gamma t)d\beta\\
    &+\sum_{j'}\partial_{\alpha}^{b_{j'}}\lambda(\alpha)\int_{-\pi}^{\pi}K_{-1}^{j'}(V_{g}^{[\frac{j+1}{2}]}(\alpha)-V_{g}^{[\frac{j+1}{2}]}(\beta),V_{f}^{[\frac{j+1}{2}]}(\alpha)-V_{f}^{[\frac{j+1}{2}]}(\beta),\tilde{V}_{g}^{[\frac{j+1}{2}]}(\alpha)-\tilde{V}_{g}^{[\frac{j+1}{2}]}(\beta))\\
    &\cdot X_i(\beta)(\tilde{z}^{j}(\alpha)-\tilde{z}^{j}(\beta))d\beta\\
    &=Term_{2,1}+Term_{2,2}.
\end{align*}
Here  $\tilde{z}^{j}\in V_g^{j} \cup  \tilde{V}_g^{j}\cup V_{\tilde{f}}^j$. $[\frac{j+1}{2}]$ is the biggest integer less than $\frac{j+1}{2}$.
Then from  \eqref{kgcondition}, we could use lemma \ref{goodterm1} to bound $Term_{2,1}$. 
Moreover, since $j+1-[\frac{j+1}{2}]\geq[\frac{j+1}{2}]\geq2 $. We have
\begin{align*}
    &\|K_{-1}^{j'}(V_{g}^{[\frac{j+1}{2}]}(\alpha)-V_{g}^{[\frac{j+1}{2}]}(\beta),V_{f}^{[\frac{j+1}{2}]}(\alpha)-V_{f}^{[\frac{j+1}{2}]}(\beta),\tilde{V}_{g}^{[\frac{j+1}{2}]}(\alpha)-\tilde{V}_{g}^{[\frac{j+1}{2}]}(\beta))(\alpha-\beta)\|_{C^0([-2\delta,2\delta]\times[-\pi,\pi])}\\
    &\lesssim C(\|g\|_{H^{j+1}}\|\tilde{f}\|_{H^{j+1}}).
\end{align*}
Then we could use lemma \ref{goodterm3} to bound $Term_{2,2}$.
\end{proof}
Then from lemma \ref{Tproperty}, \eqref{zspace} and \eqref{integral equation}, we have 
\begin{equation}\label{zspace2}
z(\alpha,\gamma,t)\in L^\infty([0,t_1],H^5_{\alpha}(\mathbb{T}))\cap C^0([0,t_1],H^4_{\alpha}(\mathbb{T}))\cap C^1([0,t_1],H^3_{\alpha}(\mathbb{T})).
\end{equation}
\section{The uniqueness}\label{uniquenessfar}
In this section we show there exists sufficiently $0<t_2\leq t_1$ such that for $0\leq t\leq t_2$, we have $z(\alpha,0,t)=f^c(\alpha,t)$.

Let $z^0(\alpha,t)=z(\alpha,0,t)$. From \eqref{zequationsimplify} and \eqref{kequationnew}, we have
\begin{align*}
    &\frac{dz_{\mu}^0(\alpha,t)}{dt}=\lambda(\alpha)\int_{-\pi}^{\pi}K(z^0(\alpha)-z^{0}(\beta),\tilde{f}(\alpha)-\tilde{f}(\beta))(\partial_{\alpha}z^0_{\mu}(\alpha)-\partial_{\beta}z^0_{\mu}(\beta))d\beta\\
    &+\lambda(\alpha)\int_{-\pi}^{\pi}K(z^0(\alpha)-z^{0}(\beta),\tilde{f}(\alpha)-\tilde{f}(\beta))(\partial_{\alpha}\tilde{f}_{\mu}(\alpha)-\partial_{\beta}\tilde{f}_{\mu}(\beta))d\beta.
\end{align*}
Moreover, from \eqref{fequation}, we have
\begin{align*}
&\frac{df^c_{\mu}(\alpha,t)}{dt}=\lambda(\alpha)\int_{-\pi}^{\pi}K(f^c(\alpha)-f^{c}(\beta),\tilde{f}(\alpha)-\tilde{f}(\beta))(\partial_{\alpha}f^c_{\mu}(\alpha)-\partial_{\beta}f^c_{\mu}(\beta))d\beta\\
    &+\lambda(\alpha)\int_{-\pi}^{\pi}K(f^c(\alpha)-f^{c}(\beta),\tilde{f}(\alpha)-\tilde{f}(\beta))(\partial_{\alpha}\tilde{f}_{\mu}(\alpha)-\partial_{\beta}\tilde{f}_{\mu}(\beta))d\beta.
\end{align*}
Then we have the equation for the difference:
\begin{align*}
    &\frac{d(z_{\mu}^0(\alpha,t)-f^c(\alpha,t))}{dt}=\lambda(\alpha)\int_{-\pi}^{\pi}K(z^0(\alpha)-z^{0}(\beta),\tilde{f}(\alpha)-\tilde{f}(\beta))(\partial_{\alpha}(z_{\mu}^0-f^c_{\mu})(\alpha)-\partial_{\beta}(z_{\mu}^0-f^c_{\mu})(\beta))d\beta\\
    &+\lambda(\alpha)\int_{-\pi}^{\pi}(K(z^0(\alpha)-z^{0}(\beta),\tilde{f}(\alpha)-\tilde{f}(\beta))-K(f^c(\alpha)-f^{c}(\beta),\tilde{f}(\alpha)-\tilde{f}(\beta)))\\
    &\cdot(\partial_{\alpha}\tilde{f}_{\mu}(\alpha)+\partial_{\alpha}f^c_{\mu}(\alpha)-\partial_{\beta}\tilde{f}_{\mu}(\beta)-\partial_{\beta}f^c_{\mu}(\beta))d\beta\\
    &=Term_{1}+Term_{2}.
\end{align*}
We first control $Term_2$, we have
\begin{align}\label{Term2bound}
    &Term_2=\\\nonumber
    &\lambda(\alpha)\int_{-\pi}^{\pi}\int_{0}^1\frac{d}{d\tau}(K(f^c(\alpha)-f^{c}(\beta)-\tau(f^c(\alpha)-z^0(\alpha)-(f^c(\beta)-z^0(\beta))),\tilde{f}(\alpha)-\tilde{f}(\beta)))\\\nonumber
    &\cdot(\partial_{\alpha}\tilde{f}_{\mu}(\alpha)+\partial_{\alpha}f^c_{\mu}(\alpha)-\partial_{\beta}\tilde{f}_{\mu}(\beta)-\partial_{\beta}f^c_{\mu}(\beta))d\tau d\beta\\\nonumber
    &=-\lambda(\alpha)\int_{-\pi}^{\pi}\int_{0}^1\nabla_1 K(f^c(\alpha)-f^{c}(\beta)-\tau(f^c(\alpha)-z^0(\alpha)-(f^c(\beta)-z^0(\beta))),\tilde{f}(\alpha)-\tilde{f}(\beta)))\\\nonumber
    &\cdot(f^c(\alpha)-z^0(\alpha)-(f^c(\beta)-z^0(\beta)))(\partial_{\alpha}\tilde{f}_{\mu}(\alpha)+\partial_{\alpha}f^c_{\mu}(\alpha)-\partial_{\beta}\tilde{f}_{\mu}(\beta)-\partial_{\beta}f^c_{\mu}(\beta))d\tau d\beta.
\end{align}
Since the component of $\nabla K$ is of $-2$ type, we have
\begin{align}\label{nablaK}
   &\|\nabla_1 K(f^c(\alpha)-f^{c}(\beta)-\tau(f^c(\alpha)-z^0(\alpha)-(f^c(\beta)-z^0(\beta))),\tilde{f}(\alpha)-\tilde{f}(\beta)))(\alpha-\beta)^2\|_{C^1_{\alpha,\beta}([-2\delta,2\delta]\times[-\pi,\pi])}\\\nonumber
   &\lesssim(\|f^c(\alpha)\|_{C^2[-\pi,\pi]}+\|z^0(\alpha)\|_{C^2[-\pi,\pi]}+\|\tilde{f}(\alpha)\|_{C^2[-\pi,\pi]})C(\|f^c-\tau(f^c-z^0)\|_{Arc}).
\end{align}
When $t=0$, we have $z^0=f^c$, then
\begin{align}\label{t0arcchord}
\|f^c-\tau(f^c-z^0)\|_{Arc}=\|f^c\|_{Arc}\lesssim 1.
\end{align}
Moreover, we have the following lemma
\begin{lemma}\label{Arcchorddiff}
For $g,h\in C^1(\mathbb{T})$, $\|h\|_{Arc}< \infty$, there exists $\delta$ depending on $\|h\|_{Arc}$ and $\|h\|_{C^1(\mathbb{T})}$ such that when $\|g-h\|_{C^1(\mathbb{T})}\leq \delta$, we have $\|g\|_{Arc}<\infty.$
\end{lemma}
\begin{proof}
We have
\begin{align*}
    &|\cosh(g_2(\alpha)-g_2(\beta)+\tilde{f}_2(\alpha)-\tilde{f}_2(\beta))-\cosh(h_2(\alpha)-h_2(\beta)+\tilde{f}_2(\alpha)-\tilde{f}_2(\beta))|\\
    &\leq|g_2(\alpha)-g_2(\beta)-(h_2(\alpha)-h_2(\beta))|\int_{0}^{1}|\sinh((1-\tau)(h_2(\alpha)-h_2(\beta))+\tau(g_2(\alpha)-g_2(\beta))+\tilde{f}_2(\alpha)-\tilde{f}_2(\beta))|d\tau\\
    &\leq (\alpha-\beta)^2\|g-h\|_{C^1(T)}C(\|h\|_{C^1(T)}+\|g-h\|_{C^1(T)}),
\end{align*}
and
\begin{align*}
    &|\cos(g_1(\alpha)-g_1(\beta)+\tilde{f}_1(\alpha)-\tilde{f}_1(\beta))-\cos(h_1(\alpha)-h_1(\beta)+\tilde{f}_1(\alpha)-\tilde{f}_1(\beta))|\\
    &\leq|g_1(\alpha)-g_1(\beta)-(h_1(\alpha)-h_1(\beta))|\int_{0}^{1}|\sin((1-\tau)(h_1(\alpha)-h_1(\beta))+\tau(g_1(\alpha)-g_1(\beta))+\tilde{f}_1(\alpha)-\tilde{f}_1(\beta))|d\tau\\
    &\leq (\alpha-\beta)^2\|g-h\|_{C^1(T)}C(\|h\|_{C^1(T)}+\|g-h\|_{C^1(T)}).
\end{align*}
Since 
\[
|\cosh(h_2(\alpha)-h_2(\beta)+\tilde{f}_2(\alpha)-\tilde{f}_2(\beta))-\cos(h_1(\alpha)-h_1(\beta)+\tilde{f}_1(\alpha)-\tilde{f}_1(\beta))|\geq \|h\|_{Arc}|\alpha-\beta|^2,
\]
we have 
\begin{align*}
&|\cosh(g_2(\alpha)-g_2(\beta)+\tilde{f}_2(\alpha)-\tilde{f}_2(\beta))-\cos(g_1(\alpha)-g_1(\beta)+\tilde{f}_1(\alpha)-\tilde{f}_1(\beta))|\\
&\geq (\|h\|_{Arc}-\|g-h\|_{C^1(T)}C(\|h\|_{C^1(T)}+\|g-h\|_{C^1(T)}))|\alpha-\beta|^2. 
\end{align*}
Then we have the result.
\end{proof}
Since we have $z^0(\alpha,t)\in C^{1}([0,t_1],H^3(\mathbb{T}))$, $f^c(\alpha,t)\in C^{1}([0,t_1],H^6(\mathbb{T}))$, then from \eqref{t0arcchord}, and lemma \ref{Arcchorddiff}, there exists $t_2$, satisfying $0\leq t_2\leq t_1$, such that for $0\leq t\leq t_2$,
\begin{align}\label{arcchordf}
\|f^c-\tau(f^c-z^0)\|_{Arc}\lesssim 1.
\end{align}
Then from corollary \ref{goodterm2}, \eqref{Term2bound} ,\eqref{nablaK}, and \eqref{arcchordf}, we have
\[
\|Term_2\|^2_{L^2(\mathbb{T})}\leq\|z^0-f^c\|^2_{L^2(\mathbb{T})}.
\]
Then 
\begin{align*}
    &\frac{d}{dt}\int_{-\pi}^{\pi}|z^0(\alpha,t)-f^c(\alpha,t)|^2d\alpha\\
    &=2\Re\int_{-\pi}^{\pi}(z^0(\alpha,t)-f^c(\alpha,t))\overline{\frac{d}{dt}(z^0(\alpha,t)-f^c(\alpha,t))}d\alpha\\
    &=\sum_{\mu=1,2}2\Re\int_{-\pi}^{\pi}(z^0_{\mu}(\alpha,t)-f^c_{\mu}(\alpha,t))\overline{\lambda(\alpha)\int_{-\pi}^{\pi}K(z^0(\alpha)-z^{0}(\beta),\tilde{f}(\alpha)-\tilde{f}(\beta))(\partial_{\alpha}(z_{\mu}^0-f^c_{\mu})(\alpha)-\partial_{\beta}(z_{\mu}^0-f^c_{\mu})(\beta))d\beta}d\alpha\\
    &+B.T^0,
\end{align*}
where $B.T^0\lesssim\|z^0(\alpha,t)-f^c(\alpha,t)\|_{L^2[-\pi,\pi]}^2$.
Then from corollary \ref{maintermcorollary} when $\gamma=0$, conditions \eqref{RTconditions}, \eqref{zspace}, \eqref{arcconditions}, we have
\[
\frac{d}{dt}\int_{-\pi}^{\pi}|z^0(\alpha,t)-f^c(\alpha,t)|^2d\alpha\lesssim\|z^0(\alpha,t)-f^c(\alpha,t)\|_{L^2(\mathbb{T})}^2.
\]
Moreover, we have $z^0(\alpha,0)=f^c(\alpha,0)$. Therefore we have 
\begin{align}\label{uniqueness}
z^0(\alpha,t)=f^c(\alpha,t),
\end{align}
for $0\leq t\leq t_2$.

 \section{The continuity of z with respect to $\gamma$}\label{continuityfar}
We first show  $\|z(\alpha, \gamma, t)-z(\alpha,\gamma',t)\|_{H^3(\alpha)}\lesssim |\gamma-\gamma'|$.

For the sake of simplicity, we further shrink the time $t_1$ to $\tilde{t}_1$ such that for all $0\leq t\leq \tilde{t}_1$, $\gamma,\gamma'\in[-1,1]$, $\tau\in[0,1]$, we have $
\|\tau z(\alpha,\gamma,t)+(1-\tau)z(\alpha,\gamma',t)-\tau z(\alpha,\gamma,0)-(1-\tau)z(\alpha,\gamma',0)\|_{C^1(\mathbb{T})}=\|\tau z(\alpha,\gamma,t)+(1-\tau)z(\alpha,\gamma',t)-f^c(\alpha,0)\|_{C^1(\mathbb{T})}$ is sufficiently small. Then from lemma \ref{Arcchorddiff}, we have 
\[
\|\tau z(\alpha,\gamma,t)+(1-\tau)z(\alpha,\gamma',t)\|_{Arc}<\infty.
\]
This is not necessary but helps to simplify our estimate in this section.

Now we estimate the difference, we have
\begin{align}\label{diffequation}
    &\frac{dz(\alpha,\gamma)}{dt}-\frac{dz(\alpha,\gamma')}{dt}=T(z(\alpha,\gamma,t),\gamma,t)-T(z(\alpha,\gamma',t),\gamma',t)\\\nonumber
    &=(T(z(\alpha,\gamma,t),\gamma,t)-T(z(\alpha,\gamma',t),\gamma,t))+(\int_{\gamma}^{\gamma'}(\partial_{\eta}T)(z(\alpha,\gamma',t),\eta,t)d\eta)\\\nonumber
    &=Term_1+Term_2.
\end{align}
For $Term_2$, we have
\begin{align}\label{dgammat}
    &(\partial_{\gamma}T)(z(\alpha),\gamma,t)=\frac{d}{d\gamma}(\frac{ic(\alpha)\gamma}{1+ic'(\alpha)\gamma t})\partial_{\alpha}z(\alpha)\\\nonumber
    &+\lambda(\alpha)\int_{-\pi}^{\pi}K(z(\alpha)-z(\beta),\tilde{f}(\alpha)-\tilde{f}(\beta))\frac{d}{d\gamma}(\frac{1+ic'(\beta)\gamma t}{1+ic'(\alpha)\gamma t}-1)d\beta\partial_{\alpha}z(\alpha)\\\nonumber
    &+\lambda(\alpha)\int_{-\pi}^{\pi}K(z(\alpha)-z(\beta),\tilde{f}(\alpha)-\tilde{f}(\beta))(\partial_{\alpha}\tilde{f}(\alpha)-\partial_{\beta}\tilde{f}(\beta))ic'(\beta)t d\beta\\\nonumber
    &=Term_{2,1}+Term_{2,2}+Term_{2,3}.
\end{align}
Since $z(\alpha,\gamma,t)\in L^{\infty}_{t}([0,t_0],H^5(\mathbb{T}))$, we have
\begin{align*}
    \|Term_{2,1}\|_{H^3(\mathbb{T})}\lesssim 1.    \end{align*}
    Moreover,
    \begin{align*}
        &Term_{2,2}=\lambda(\alpha)\int_{-\pi}^{\pi}K(z(\alpha)-z(\beta),\tilde{f}(\alpha)-\tilde{f}(\beta))[\frac{(ic(\beta)-ic(\alpha))t}{1+ic'(\alpha)\gamma t}-\frac{(ic(\beta)-ic(\alpha))\gamma tic'(\alpha)t}{(1+c'(\alpha)\gamma t)^2}]d\beta\partial_{\alpha}z(\alpha).
    \end{align*}
    Since $K$ is of $-1$ type, we have
    \begin{align}\label{Kcondition}
    K(z(\alpha)-z(\beta),\tilde{f}(\alpha)-\tilde{f}(\beta))(\alpha-\beta)\in C^3([-2\delta,2\delta]\times[-\pi,\pi]).
    \end{align}
    Therefore $\|Term_{2,2}\|_{H^3(\mathbb{T})}\lesssim 1$.
    Moreover, $Term_{2,3}$ can be bounded in the similar way since $|\frac{\partial_{\alpha}\tilde{f}(\alpha)-\partial_{\beta}\tilde{f}(\beta)}{\alpha-\beta}|\in C^3([-2\delta,2\delta]\times[-\pi,\pi])$ and we get
    \[
    \|Term_{2,3}\|_{H^3(\mathbb{T})}\lesssim 1.
    \]
    Then we have
    \begin{align}\label{partialgammabound}
        \|(\partial_{\gamma}T)(z(\alpha),\gamma,t)\|_{H^3(\mathbb{T})}\lesssim 1,
    \end{align}
    and 
     \begin{align}\label{diffterm2H3}
        \|Term_2\|_{H^3(\mathbb{T})}\lesssim |\gamma-\gamma'|.
    \end{align}
For $Term_1$, notice that $\tilde{f}(\alpha+ic(\alpha)\gamma t,t)=\tilde{f}(\alpha,t)$, we have
\begin{align}\label{Diffterm1}
    &Term_1=\frac{ic(\alpha)\gamma}{1+ic'(\alpha)\gamma t}\partial_{\alpha}(z(\alpha,\gamma)-z(\alpha,\gamma'))\\\nonumber
    &+\lambda(\alpha)\int_{-\pi}^{\pi}K(z(\alpha,\gamma)-z(\beta,\gamma),\tilde{f}(\alpha)-\tilde{f}(\beta))(\frac{\partial_{\alpha}z(\alpha,\gamma)-\partial_{\alpha}z(\alpha,\gamma')}{1+ic'(\alpha)\gamma t}-(\frac{\partial_{\beta}z(\beta,\gamma)-\partial_{\beta}z(\beta,\gamma')}{1+ic'(\beta)\gamma t}))\\\nonumber
    &\cdot(1+ic'(\beta)\gamma t)d\beta\\\nonumber
    &+\lambda(\alpha)\int_{-\pi}^{\pi}(K(z(\alpha,\gamma)-z(\beta,\gamma),\tilde{f}(\alpha)-\tilde{f}(\beta))-K(z(\alpha,\gamma')-z(\beta,\gamma'),\tilde{f}(\alpha)-\tilde{f}(\beta)))\\\nonumber
    &(\frac{\partial_{\alpha}z(\alpha,\gamma')}{1+ic'(\alpha)\gamma t}-\frac{\partial_{\beta}z(\beta,\gamma')}{1+ic'(\beta)\gamma t}+\partial_{\alpha}\tilde{f}(\alpha)-\partial_{\beta}\tilde{f}(\beta))(1+ic'(\beta)\gamma t)d\beta\\\nonumber
    &=Term_{1,1}+Term_{1,2}+Term_{1,3}.
    \end{align}
It is easy to get
\begin{align}\label{diffterm12L2}
    \|Term_{1,1}\|_{L^2(T)}+\|Term_{1,2}\|_{L^2(T)}\lesssim \|z(\alpha,\gamma)-z(\alpha,\gamma')\|_{H^3(\mathbb{T})}.\end{align}
    Moreover,
\begin{align}\label{Diff13L2}
    &Term_{1,3}=\lambda(\alpha)\int_{-\pi}^{\pi}\int_{0}^{1}\nabla_1 K(\tau(z(\alpha,\gamma)-z(\beta,\gamma))+(1-\tau)(z(\alpha,\gamma')-z(\beta,\gamma')),\tilde{f}(\alpha)-\tilde{f}(\beta))d\tau\\\nonumber
    &\cdot(z(\alpha,\gamma)-z(\beta,\gamma)-(z(\alpha,\gamma')-z(\beta,\gamma')))\\\nonumber
    &\cdot(\frac{\partial_{\alpha}z(\alpha,\gamma')}{1+ic'(\alpha)\gamma t}-\frac{\partial_{\beta}z(\beta,\gamma')}{1+ic'(\beta)\gamma t}+\partial_{\alpha}\tilde{f}(\alpha)-\partial_{\beta}\tilde{f}(\beta))(1+ic'(\beta)\gamma t) d\beta.
\end{align}
Since the component of $\nabla_1 K$ is of $-2 $ type, we have 
\[\sup_{\tau}\|\nabla_1 K(\tau(z(\alpha,\gamma)-z(\beta,\gamma)+(1-\tau)(z(\alpha,\gamma')-z(\beta,\gamma')),\tilde{f}(\alpha)-\tilde{f}(\beta))(\alpha-\beta)^2\|_{C^3([-2\delta,2\delta]\times[-\pi,\pi])}\lesssim 1,
\]
then we  have
\begin{align}\label{diffterm13L2}
\|Term_{1,3}\|_{L^2(T)}\lesssim\|z(\alpha,\gamma)-z(\alpha,\gamma')\|_{H^3(T)}.
\end{align}
Now we control $\partial_{\alpha}^3 Term_1$.
For $\partial_{\alpha}^{3}Term_{1,1}$, we have
\begin{align}\label{diffterm1h3}
    &\partial_{\alpha}^{3}Term_{1,1}=\frac{ic(\alpha)\gamma}{1+ic'(\alpha)\gamma t}\partial_{\alpha}^4(z(\alpha,\gamma)-z(\alpha,\gamma'))+\sum_{1\leq j\leq 3} C_{1,j}\partial_{\alpha}^{j}(\frac{ic(\alpha)\gamma}{1+ic'(\alpha)\gamma t})\partial_{\alpha}^{4-j}(z(\alpha,\gamma)-z(\alpha,\gamma'))\\\nonumber
    &=Term_{1,1,M}^{3}+Term_{1,1,2}^3.
\end{align}
Here $\|Term_{1,1,2}^3\|_{L^2(\mathbb{T})}\lesssim\|z(\alpha,\gamma)-z(\alpha,\gamma')\|_{H^3(\mathbb{T})}$.
For $\partial_{\alpha}^{3}Term_{1,2}$, from lemma \ref{goodterm5} and  \eqref{Kcondition}, we have
\begin{align}\label{diffterm12h3}
    &\partial_{\alpha}^3 Term_{1,2}=\lambda(\alpha)\int_{-\pi}^{\pi}K(z(\alpha,\gamma)-z(\alpha-\beta,\gamma),\tilde{f}(\alpha)-\tilde{f}(\alpha-\beta))\\\nonumber
    &\cdot(\partial_{\alpha}^{3}(\frac{\partial_{\alpha}(z(\alpha,\gamma)-z(\alpha,\gamma'))}{1+ic'(\alpha)\gamma t})-\partial_{\alpha}^{3}(\frac{\partial_{\alpha}(z(\alpha-\beta,\gamma)-z(\alpha-\beta,\gamma'))}{1+ic'(\alpha-\beta)\gamma t}))(1+ic'(\alpha-\beta)\gamma t)d\beta\\\nonumber
    &+Term_{1,2,1}^3,
\end{align}
where $\|Term_{1,2,1}^3\|_{L^2(\mathbb{T})}\lesssim\|z(\alpha,\gamma)-z(\alpha,\gamma')\|_{H^3(\mathbb{T})}$.
Moreover, we have
\begin{align}\label{diffterm12h32}
      &\lambda(\alpha)\int_{-\pi}^{\pi}K(z(\alpha,\gamma)-z(\alpha-\beta,\gamma),\tilde{f}(\alpha)-\tilde{f}(\alpha-\beta))\\\nonumber
    &\cdot(\partial_{\alpha}^{3}(\frac{\partial_{\alpha}(z(\alpha,\gamma)-z(\alpha,\gamma'))}{1+ic'(\alpha)\gamma t})-\partial_{\alpha}^{3}(\frac{\partial_{\alpha}(z(\alpha-\beta,\gamma)-z(\alpha-\beta,\gamma'))}{1+ic'(\alpha-\beta)\gamma t}))(1+ic'(\alpha-\beta)\gamma t)d\beta\\\nonumber
    &= \lambda(\alpha)\int_{-\pi}^{\pi}K(z(\alpha,\gamma)-z(\alpha-\beta,\gamma),\tilde{f}(\alpha)-\tilde{f}(\alpha-\beta))\\\nonumber
    &\cdot(\frac{\partial_{\alpha}^4(z(\alpha,\gamma)-z(\alpha,\gamma'))}{1+ic'(\alpha)\gamma t}-\frac{\partial_{\alpha}^4(z(\alpha-\beta,\gamma)-z(\alpha-\beta,\gamma'))}{1+ic'(\alpha-\beta)\gamma t}))(1+ic'(\alpha-\beta)\gamma t)d\beta\\\nonumber
    &+\sum_{0\leq j\leq 2} C_j  \lambda(\alpha)\int_{-\pi}^{\pi}K(z(\alpha,\gamma)-z(\alpha-\beta,\gamma),\tilde{f}(\alpha)-\tilde{f}(\alpha-\beta))(1+ic'(\alpha-\beta)\gamma t)\\\nonumber
    &\cdot(\partial_{\alpha}^{1+j}(z(\alpha,\gamma)-z(\alpha,\gamma'))\partial_{\alpha}^{3-j}(\frac{1}{1+ic'(\alpha)\gamma t})-\partial_{\alpha}^{1+j}(z(\alpha-\beta,\gamma)-z(\alpha-\beta,\gamma'))\partial_{\alpha}^{3-j}(\frac{1}{1+ic'(\alpha-\beta)\gamma t}))d\beta\\\nonumber
    &=Term_{1,2,M}^{3}+Term_{1,2,3}^{3}.
\end{align}
Then from lemma \ref{goodterm1}, we have
\begin{align}\label{diffterm12h33}
    \|Term_{1,2,3}^{3}\|_{L^2(\mathbb{T})}\lesssim\|z(\alpha,\gamma)-z(\alpha,\gamma')\|_{H^3(\mathbb{T})}.
\end{align}
For $\partial_{\alpha}^3 Term_{1,3}$, we use equation \eqref{Diff13L2}. Since
\begin{align*}
&\|\nabla K(\tau(z(\alpha,\gamma)-z(\beta,\gamma))+(1-\tau)(z(\alpha,\gamma')-z(\beta,\gamma')),\tilde{f}(\alpha)-\tilde{f}(\beta))(\alpha-\beta)^2
\|_{C^3([-2\delta,2\delta]\times[-\pi,\pi])}\lesssim 1,
\end{align*}
and 
\begin{align*}
    \frac{\partial_{\alpha}z(\alpha,\gamma')}{1+ic'(\alpha)\gamma t}+\partial_{\alpha}\tilde{f}(\alpha)\in H^4(\mathbb{T}),
\end{align*}
from lemma \ref{goodterm6} we have
\begin{align}\label{diffterm13h3}
    \|\partial_{\alpha}^3 Term_{1,3}\|_{L^2(\mathbb{T})}\lesssim\|z(\alpha,\gamma)-z(\alpha,\gamma')\|_{H^3(\mathbb{T})}.
\end{align}
In conclusion, from \eqref{diffterm12L2}, \eqref{diffterm13L2}, and \eqref{diffterm2H3}, we have
\begin{align*}
    &\frac{d}{dt}\int_{-\pi}^{\pi}|z(\alpha,\gamma)-z(\alpha,\gamma')|^2d\alpha=2\Re \int_{-\pi}^{\pi} (z(\alpha,\gamma)-z(\alpha,\gamma'))\overline{Term_{1}+Term_{2}}d\alpha\\
    &\lesssim|\gamma-\gamma'|^2+\|z(\alpha,\gamma)-z(\alpha,\gamma')\|_{H^3(\mathbb{T})}^2.
\end{align*}
From \eqref{diffterm1h3} ,\eqref{diffterm12h3}, \eqref{diffterm12h32}, \eqref{diffterm12h33}, \eqref{diffterm13h3} and \eqref{diffterm2H3}, we have
\begin{align*}
&\frac{d}{dt}\int_{-\pi}^{\pi}|\partial_{\alpha}^3(z(\alpha,\gamma)-z(\alpha,\gamma'))|^2d\alpha\\
&=2\Re\int_{-\pi}^{\pi}\partial_{\alpha}^3(z(\alpha,\gamma)-z(\alpha,\gamma'))\cdot\overline{\partial_{\alpha}^3 Term_1+\partial_{\alpha}^3 Term_2}d\alpha\\
&=2\Re\int_{-\pi}^{\pi}\partial_{\alpha}^3(z(\alpha,\gamma)-z(\alpha,\gamma'))\cdot\overline{ Term_{1,1,M}^3+Term_{1,2,M}^3}d\alpha+B.T^0\\
&=\sum_{\mu=1,2}2\Re\int_{-\pi}^{\pi}\partial_{\alpha}^3(z_{\mu}(\alpha,\gamma)-z_{\mu}(\alpha,\gamma'))\overline{\frac{ic(\alpha)\gamma}{1+ic'(\alpha)\gamma t}\partial_{\alpha}^4(z_{\mu}(\alpha,\gamma)-z(\alpha,\gamma'))}d\alpha\\
&+\sum_{\mu=1,2}2\Re\int_{-\pi}^{\pi}\partial_{\alpha}^3(z_{\mu}(\alpha,\gamma)-z_{\mu}(\alpha,\gamma'))\overline{\lambda(\alpha)\int_{-\pi}^{\pi}K(z(\alpha,\gamma)-z(\alpha-\beta,\gamma),\tilde{f}(\alpha)-\tilde{f}(\alpha-\beta))}\\\nonumber
    &\overline{\cdot(\frac{\partial_{\alpha}^4(z(\alpha,\gamma)-z(\alpha,\gamma'))}{1+ic'(\alpha)\gamma t}-\frac{\partial_{\alpha}^4(z(\alpha-\beta,\gamma)-z(\alpha-\beta,\gamma'))}{1+ic'(\alpha-\beta)\gamma t}))(1+ic'(\alpha-\beta)\gamma t)d\beta}d\alpha\\
    &+B.T^0,
\end{align*}
where
\[
|B.T^0|\lesssim \|z(\alpha,\gamma)-z(\alpha,\gamma')\|_{H^3(\mathbb{T})}^2+|\gamma-\gamma'|^2.
\]
Then from corollary \ref{maintermcorollary}, we have
\begin{align*}
 & \frac{d}{dt}\int_{-\pi}^{\pi}|\partial_{\alpha}^3(z(\alpha,\gamma)-z(\alpha,\gamma'))|^2d\alpha\lesssim \|z(\alpha,\gamma)-z(\alpha,\gamma')\|_{H^3(\mathbb{T})}^2+|\gamma-\gamma'|^2.
\end{align*}
Moreover, the initial date $\|z(\alpha,\gamma)-z(\alpha,\gamma')\|_{H^3(\mathbb{T})}^2|_{t=0}=0$.
Therefore we have
\begin{align}\label{gammacontinuity}
    \|\frac{z(\alpha,\gamma)-z(\alpha,\gamma')}{\gamma-\gamma'}\|_{H^3(\mathbb{T})}\lesssim 1.
\end{align}
\section{The differentiability of z with respect to $\gamma$}\label{diffrenfar}
Now we show the differentiability. We define a new function $w(\alpha,\gamma,t)$. It satisfies the equation that $\frac{dz}{d\gamma}$ would satisfy if it is differentiable. 

Let $w$ be the solution of the following equation:
\begin{align}\label{Wequation} 
    \frac{dw(\alpha,\gamma,t)}{dt}=\tilde{T}(w)=D_zT(z(\alpha,\gamma,t),\gamma,t)[w]+\partial_{\gamma}T(\alpha,\gamma,t),
\end{align}
with initial value $w(\alpha,\gamma,0)=0$.
Here $D_z T(z(\alpha,\gamma,t),\gamma,t)[w]$ is the Gateaux derivative.

As in the existence of $z(\alpha,\gamma,t)$, we first show the energy estimate.
First, from \eqref{partialgammabound}, we have
  \begin{align*}
        \|(\partial_{\gamma}T)(z(\alpha),\gamma,t)\|_{H^3(\mathbb{T})}\lesssim 1.
    \end{align*}
    Moreover,
\begin{align}\label{dzw}
    &D_zT(z(\alpha,\gamma,t),\gamma,t)[w]=\frac{ic(\alpha)\gamma}{1+ic^{'}(\alpha)\gamma t}\partial_{\alpha}w(\alpha,\gamma)+\\\nonumber
&\lambda(\alpha)\int_{-\pi}^{\pi} K(z(\alpha,\gamma)-z(\beta,\gamma),\tilde{f}(\alpha)-\tilde{f}(\beta))(\frac{\partial_{\alpha}w(\alpha,\gamma)}{1+ic'(\alpha)\gamma t}-\frac{\partial_{\beta}w(\beta,\gamma)}{1+ic'(\beta)\gamma t})(1+ic'(\beta)\gamma t)d\beta\\\nonumber
&+\lambda(\alpha)\int_{-\pi}^{\pi} (\nabla_1 K(z(\alpha,\gamma)-z(\beta,\gamma),\tilde{f}(\alpha)-\tilde{f}(\beta))\cdot(w(\alpha,\gamma)-w(\beta,\gamma)))\\\nonumber
&\cdot(\partial_{\alpha}\tilde{f}(\alpha)-\partial_{\beta}\tilde{f}(\beta)+\frac{\partial_{\alpha}z(\alpha,\gamma)}{1+ic'(\alpha)\gamma t}-\frac{\partial_{\beta}z(\beta,\gamma)}{1+ic'(\beta)\gamma t})(1+ic'(\beta)\gamma t)d\beta.
\end{align}
It has the similar structure as \eqref{Diffterm1} and \eqref{Diff13L2}. The only difference between the first two terms in \eqref{Diffterm1} and \eqref{dzw} is that $\partial_{\alpha} w(\alpha,\gamma)$ takes the place of $\frac{z(\alpha,\gamma)-z(\alpha,\gamma')}{\gamma-\gamma'}.$ In \eqref{Diff13L2}, and the third term of $\eqref{Diffterm1}$, $w(\alpha,\gamma)$ takes the place of $\frac{z(\alpha,\gamma)-z(\alpha,\gamma')}{\gamma-\gamma'}$ and $\int_{0}^{1}\nabla_1 K(\tau(z(\alpha,\gamma)-z(\beta,\gamma))+(1-\tau)(z(\alpha,\gamma')-z(\beta,\gamma')),\tilde{f}(\alpha)-\tilde{f}(\beta))d\tau$ is replaced by $\nabla_1 K(z(\alpha,\gamma)-z(\beta,\gamma),\tilde{f}(\alpha)-\tilde{f}(\beta))$. Therefore we could use the similar estimate and have
\[
\frac{d}{dt}\|w\|^2_{H^3(\mathbb{T})}\lesssim C(\|w\|^2_{H^3(\mathbb{T})}).
\]
As in the existence of $z(\alpha,\gamma,t)$, we could do the similar energy estimate to the approximation of the equation
\begin{align}\label{Wequationappro} 
    \frac{dw^n(\alpha,\gamma,t)}{dt}=\tilde{T}(w)=\varphi_{n}*(D_zT(z(\alpha,\gamma,t),\gamma,t)[\varphi_n*w^n])+\varphi_{n}*(\partial_{\gamma}T(\alpha,\gamma,t)),
\end{align}
with initial value $w^n(\alpha,\gamma,0)=0$.
Then from the Picard theorem and  compactness argument, there exists $0\leq t_3\leq t_1$, such that 
\[
w(\alpha,\gamma,t)=\int_{0}^{t}\tilde{T}(w(\alpha,\gamma,\tau))d\tau,
\]
and 
\begin{align}\label{wLinfty}
    \|w(\alpha,\gamma,t)\|_{L^{\infty}([0,t_3],H^3_{\alpha}(\mathbb{T}))}\lesssim 1.
\end{align}
We claim there is an uniform $t_3$ holds for all $\gamma\in[-1,1]$.
Moreover, we have the following lemma:
   \begin{lemma}\label{tildeTproperty}
    For any $g(\alpha),h(\alpha)\in H^{j+1}(\mathbb{T})$, $j\leq 2$, we have
    \begin{align}\label{upperbound}
    \|\tilde{T}(g(\alpha),\gamma,t)\|_{H^j(\mathbb{T})}\lesssim 1,
    \end{align}
     \begin{align}
    \|\tilde{T}(g(\alpha),\gamma,t)-\tilde{T}(h(\alpha),\gamma,t)\|_{H^j(\mathbb{T})}\lesssim \|g(\alpha)-h(\alpha)\|_{H^{j+1}(\mathbb{T})},
    \end{align}
     \begin{align}
   \lim_{t\to t'} \|\tilde{T}(g(\alpha),\gamma,t)-\tilde{T}(g(\alpha),\gamma,t')\|_{H^j(\mathbb{T})}=0 .
    \end{align}
\end{lemma}
\begin{proof}
It is easy to get these bounds since $z(\alpha,\gamma,t)\in L^{\infty}([0,t_0], H_{\alpha}^5(\mathbb{T}))\cap C^{0}([0,t_0], H_{\alpha}^4(\mathbb{T})).$
\end{proof}
Then we have $w(\alpha,\gamma,t)\in  L^{\infty}([0,t_3], H_{\alpha}^3(\mathbb{T}))\cap C^{0}([0,t_3], H_{\alpha}^2(\mathbb{T}))\cap C^{1}([0,t_3], H_{\alpha}^1(\mathbb{T})).$ 

We claim that we could do the similar argument as in the estimate of $\|z(\alpha,\gamma)-z(\alpha,\gamma')\|_{H^3_{\alpha}(\mathbb{T})}\lesssim|\gamma-\gamma'|$ to get
\[
\|w(\alpha,\gamma)-w(\alpha,\gamma)\|_{H^1_{\alpha}(\mathbb{T})}\lesssim |\gamma-\gamma'|.
\]
Then from \eqref{wLinfty}, we have
\begin{align}\label{wcontinuous}
\lim_{\gamma'\to\gamma}\|w(\alpha,\gamma')-w(\alpha,\gamma)\|_{H^2_{\alpha}(\mathbb{T})}=0.
\end{align}
Now we show $w$ is the derivative of $z$ with respect of $\gamma$. Let
\[
v(\alpha,\gamma,\gamma',t)=\frac{z(\alpha,\gamma,t)-z(\alpha,\gamma',t)}{\gamma-\gamma'}-w(\alpha,\gamma,t).
\]
From \eqref{Wequation} and \eqref{diffequation}, we have
\begin{align*}
&\frac{dv(\alpha,\gamma,\gamma',t)}{dt}=\frac{T(z(\alpha,\gamma,t),\gamma,t)-T(z(\alpha,\gamma',t),\gamma,t)}{\gamma-\gamma'}+\frac{\int_{\gamma}^{\gamma'}(\partial_{\eta}T)(z(\alpha,\gamma',t),\eta,t)d\eta}{\gamma-\gamma'}\\
&-D_zT(z(\alpha,\gamma,t),\gamma,t)[w(\alpha,\gamma,t)]-(\partial_{\gamma}T)(z(\alpha,\gamma,t),\gamma,t)\\
&=(\frac{T(z(\alpha,\gamma,t),\gamma,t)-T(z(\alpha,\gamma',t),\gamma,t)}{\gamma-\gamma'}-D_zT(z(\alpha,\gamma,t),\gamma,t)[w(\alpha,\gamma,t)])\\
&+(\frac{\int_{\gamma}^{\gamma'}(\partial_{\eta}T)(z(\alpha,\gamma',t),\eta,t)d\eta}{\gamma-\gamma'}-(\partial_{\gamma}T)(z(\alpha,\gamma,t),\gamma,t))\\
&=Term_1+Term_2.
\end{align*}
We have 
\begin{align*}
    &Term_2=\frac{1}{\gamma-\gamma'}\int_{\gamma}^{\gamma'}(\partial_{\eta}T)(z(\alpha,\gamma',t),\eta,t)d\eta-(\partial_{\gamma}T)(z(\alpha,\gamma,t),\gamma,t)\\
    &=\int_{\gamma}^{\gamma'}\frac{(\partial_{\eta}T)(z(\alpha,\gamma',t),\eta,t)-(\partial_{\gamma}T)(z(\alpha,\gamma',t),\gamma,t)}{\gamma-\gamma'}d\eta+((\partial_{\gamma}T)(z(\alpha,\gamma',t),\gamma,t)-(\partial_{\gamma}T)(z(\alpha,\gamma,t),\gamma,t)).
\end{align*}
From \eqref{dgammat} and \eqref{gammacontinuity}, we have
\[
\|((\partial_{\gamma}T)(z(\alpha,\gamma',t),\gamma,t)-(\partial_{\gamma}T)(z(\alpha,\gamma,t),\gamma,t))\|_{L^2(\mathbb{T})}\lesssim\|z(\alpha,\gamma,t)-z(\alpha,\gamma',t)\|_{H^2(\mathbb{T})}\lesssim |\gamma-\gamma'|,
\]
and
\[
\|(\partial_{\eta}T)(z(\alpha,\gamma',t),\eta,t)-(\partial_{\gamma}T)(z(\alpha,\gamma',t),\gamma,t)\|_{L^2(\mathbb{T})}\lesssim |\eta-\gamma|,
\]
Then 
\[
\|Term_{2}\|_{L^2(\mathbb{T})}\lesssim |\gamma-\gamma'|.
\]
Moreover, for $Term_1$, from \eqref{dzw}, \eqref{diffequation}, \eqref{Diffterm1} and \eqref{Diff13L2} we have
\begin{align*}
    &\frac{T(z(\alpha,\gamma,t),\gamma,t)-T(z(\alpha,\gamma',t),\gamma,t)}{\gamma-\gamma'}-D_zT(z(\alpha,\gamma,t),\gamma,t)[w(\alpha,\gamma,t)]\\
    &=\frac{ic(\alpha)\gamma}{1+ic'(\alpha)\gamma t}\partial_{\alpha}(v(\alpha,\gamma,\gamma'))\\
    &+\lambda(\alpha)\int_{-\pi}^{\pi}K(z(\alpha,\gamma)-z(\beta,\gamma),\tilde{f}(\alpha)-\tilde{f}(\beta))(\frac{\partial_{\alpha}v(\alpha,\gamma,\gamma')}{1+ic'(\alpha)\gamma t}-\frac{\partial_{\beta}v(\beta,\gamma,\gamma')}{1+ic'(\beta)\gamma t})(1+ic'(\beta)\gamma t)d\beta\\
    &+ \lambda(\alpha)\int_{-\pi}^{\pi}\int_{0}^{1}\nabla_1 K(\tau(z(\alpha,\gamma)-z(\beta,\gamma))+(1-\tau)(z(\alpha,\gamma')-z(\beta,\gamma')),\tilde{f}(\alpha)-\tilde{f}(\beta))\\\nonumber
    &\cdot(\frac{z(\alpha,\gamma)-z(\beta,\gamma)-(z(\alpha,\gamma')-z(\beta,\gamma'))}{\gamma-\gamma'})\\\nonumber
    &(\frac{\partial_{\alpha}z(\alpha,\gamma')}{1+ic'(\alpha)\gamma t}-\frac{\partial_{\beta}z(\beta,\gamma')}{1+ic'(\beta)\gamma t}+\partial_{\alpha}\tilde{f}(\alpha)-\partial_{\beta}\tilde{f}(\beta))(1+ic'(\beta)\gamma t)d\tau d\beta\\
    &-\lambda(\alpha)\int_{-\pi}^{\pi} \nabla_1 K(z(\alpha,\gamma)-z(\beta,\gamma),\tilde{f}(\alpha)-\tilde{f}(\beta))\cdot(w(\alpha,\gamma)-w(\beta,\gamma))\\\nonumber
&\cdot(\partial_{\alpha}\tilde{f}(\alpha)-\partial_{\beta}\tilde{f}(\beta)+\frac{\partial_{\alpha}z(\alpha,\gamma)}{1+ic'(\alpha)\gamma t}-\frac{\partial_{\beta}z(\beta,\gamma)}{1+ic'(\beta)\gamma t})(1+ic'(\beta)\gamma t)d\beta\\
    &=Term_{1,1}+Term_{1,2}+Term_{1,3}+Term_{1,4}.
\end{align*}
Here 
\begin{align*}
    &Term_{1,3}+Term_{1,4}\\
    &=\lambda(\alpha)\int_{-\pi}^{\pi} \int_{0}^{1}\nabla_1 K(\tau(z(\alpha,\gamma)-z(\beta,\gamma))+(1-\tau)(z(\alpha,\gamma')-z(\beta,\gamma')),\tilde{f}(\alpha)-\tilde{f}(\beta))d\tau \cdot(v(\alpha,\gamma,\gamma')-v(\beta,\gamma,\gamma'))\\\nonumber
&\cdot(\partial_{\alpha}\tilde{f}(\alpha)-\partial_{\beta}\tilde{f}(\beta)+\frac{\partial_{\alpha}z(\alpha,\gamma')}{1+ic'(\alpha)\gamma t}-\frac{\partial_{\beta}z(\beta,\gamma')}{1+ic'(\beta)\gamma t})(1+ic'(\beta)\gamma t) d\beta\\
&+\lambda(\alpha)\int_{-\pi}^{\pi}[\int_{0}^{1}\nabla_1 K(\tau(z(\alpha,\gamma)-z(\beta,\gamma))+(1-\tau)(z(\alpha,\gamma')-z(\beta,\gamma')),\tilde{f}(\alpha)-\tilde{f}(\beta))d\tau\\
&-\nabla_1 K(z(\alpha,\gamma)-z(\beta,\gamma),\tilde{f}(\alpha)-\tilde{f}(\beta))]\cdot(w(\alpha,\gamma)-w(\beta,\gamma))\\\nonumber
    &(\frac{\partial_{\alpha}z(\alpha,\gamma')}{1+ic'(\alpha)\gamma t}-\frac{\partial_{\beta}z(\beta,\gamma')}{1+ic'(\beta)\gamma t}+\partial_{\alpha}\tilde{f}(\alpha)-\partial_{\beta}\tilde{f}(\beta))(1+ic'(\beta)\gamma t) d\beta\\
    &+\lambda(\alpha)\int_{-\pi}^{\pi}\nabla_1 K(z(\alpha,\gamma)-z(\beta,\gamma),\tilde{f}(\alpha)-\tilde{f}(\beta))\cdot(w(\alpha,\gamma)-w(\beta,\gamma))\\\nonumber
    &(\frac{\partial_{\alpha}z(\alpha,\gamma')}{1+ic'(\alpha)\gamma t}-\frac{\partial_{\beta}z(\beta,\gamma')}{1+ic'(\beta)\gamma t}-(\frac{\partial_{\alpha}z(\alpha,\gamma)}{1+ic'(\alpha)\gamma t}-\frac{\partial_{\beta}z(\beta,\gamma)}{1+ic'(\beta)\gamma t}))(1+ic'(\beta)\gamma t) d\beta\\
    &=Term_{1,5}+Term_{1,6}+Term_{1,7}.
\end{align*}
Since the component of $\nabla_1 K$ is of $-2$ type, we could use lemma $\ref{goodterm1}$ to bound  $Term_{1,5}$ and have
\begin{align*}
    \|Term_{1,5}\|_{L^2(\mathbb{T})}^2\lesssim\|v(\alpha,\gamma,\gamma')\|_{L^2(\mathbb{T})}^2.
\end{align*}
For $Term_{1,6}$, we have
\begin{align*}
    &\|[\int_{0}^{1}\nabla_1 K(\tau(z(\alpha,\gamma)-z(\beta,\gamma))+(1-\tau)(z(\alpha,\gamma')-z(\beta,\gamma')),\tilde{f}(\alpha)-\tilde{f}(\beta))d\tau\\
&-\nabla_1 K(z(\alpha,\gamma)-z(\beta,\gamma),\tilde{f}(\alpha)-\tilde{f}(\beta))](\alpha-\beta)^2\|_{C^1([-2\delta,\delta]\times [-\pi,\pi])}\\
&\lesssim \|z(\alpha,\gamma)-z(\alpha,\gamma')\|_{H^3(\mathbb{T})}.
\end{align*}
Then from lemma \ref{goodterm2}, and \eqref{gammacontinuity}, we have
\begin{align*}
    \|Term_{1,6}\|_{L^2(\mathbb{T})}\lesssim |\gamma-\gamma'|.
\end{align*}
From lemma \eqref{goodterm2}, we again have
\begin{align*}
    &\|Term_{1,7}\|_{L^2(\mathbb{T})}\\
    &\lesssim\|\nabla_1 K(z(\alpha,\gamma)-z(\beta,\gamma),\tilde{f}(\alpha)-\tilde{f}(\beta))(\alpha-\beta)^2\|_{C^1([-2\delta,\delta]\times [-\pi,\pi])}\|w(\alpha,\gamma)\|_{C^2(\mathbb{T})}\|\frac{\partial_{\alpha}z(\alpha,\gamma')}{1+ic'(\alpha)\gamma t}-\frac{\partial_{\alpha}z(\alpha,\gamma)}{1+ic'(\alpha)\gamma t}\|_{L^2(\mathbb{T})}\\
    &\lesssim \|z(\alpha,\gamma)-z(\alpha,\gamma')\|_{H^1(\mathbb{T})}\lesssim |\gamma-\gamma'|,
\end{align*}
where we use \eqref{wLinfty} and \eqref{gammacontinuity}.
Therefore we have
\begin{align*}
    &\frac{d}{dt}\int_{-\pi}^{\pi}|v(\alpha,\gamma,\gamma')|^2d\alpha\\
    &=2\Re\int_{-\pi}^{\pi}v(\alpha,\gamma,\gamma')\overline{\frac{ic(\alpha)\gamma}{1+ic'(\alpha)\gamma t}\partial_{\alpha}(v(\alpha,\gamma,\gamma'))}d\alpha\\
    &+2\Re\int_{-\pi}^{\pi}v(\alpha,\gamma,\gamma')\overline{\lambda(\alpha)\int_{-\pi}^{\pi}K(z(\alpha,\gamma)-z(\beta,\gamma),\tilde{f}(\alpha)-\tilde{f}(\beta))(\frac{\partial_{\alpha}v(\alpha,\gamma,\gamma')}{1+ic'(\alpha)\gamma t}-\frac{\partial_{\beta}v(\beta,\gamma,\gamma')}{1+ic'(\beta)\gamma t})(1+ic'(\beta)\gamma t)d\beta}d\alpha\\
    &+B.T^0,
\end{align*}
where
\[
B.T^0\lesssim |\gamma-\gamma'|^2+\|v(\alpha,\gamma,\gamma')\|_{L^2(\mathbb{T})}^2.
\]
Then by corollary \ref{maintermcorollary}, and the initial value $v(\alpha,\gamma,\gamma',0)=0$, from the Gronwall's inequality we have
$\lim_{\gamma'\to\gamma}\|v(\alpha,\gamma,\gamma',t)\|_{L^2(\mathbb{T})}=0$ when $t\leq t_3$.

Form \eqref{wLinfty}, and \eqref{gammacontinuity}, we have
\begin{align*}
    \|v(\alpha,\gamma,\gamma',t)\|_{H^3(\mathbb{T})}\lesssim 1.
\end{align*}
Moreover from the interpolation theorem, we have
\begin{align*}
    \lim_{\gamma'\to\gamma}\|v(\alpha,\gamma,\gamma',t)\|_{H^2(\mathbb{T})}=0.
\end{align*}
Then from \eqref{wcontinuous}, we have 
\[
z(\alpha,\gamma)\in C^1_\gamma([-1,1],H^2_{\alpha}(\mathbb{T})),
\]
with
\[
\frac{d}{d\gamma}z(\alpha,\gamma, t) =w(\alpha,\gamma).
\]
From \eqref{Wequation}, we also have
\begin{align*}
    \frac{d}{dt}\frac{d}{d\gamma}z=\frac{dw}{dt}=\frac{dT(z)}{d\gamma}=\frac{d}{d\gamma}\frac{d}{dt}z.
\end{align*}
In conclusion we have 

\begin{align}\label{zconditions}
\begin{cases}
&z(\alpha,\gamma,t)\in C^1_{t}([0,t_1], H_{\alpha}^3(\mathbb{T})),\\
&z(\alpha,\gamma,t)\in C^1_\gamma([-1,1],H^2_{\alpha}(\mathbb{T})),\\
&\frac{dz}{d\gamma}\in C^{0}_t([0,t_3], H_{\alpha}^2(\mathbb{T}))\cap C^{1}_t([0,t_3], H_{\alpha}^1(\mathbb{T})),\\
&\frac{d}{dt}\frac{d}{d\gamma}z=\frac{d}{d\gamma}\frac{d}{dt}z,\\
&z(\alpha,0,t)=f^c(\alpha,t), \text{ when }0\leq t\leq t_2.
\end{cases}
\end{align}
\section{The analyticity}\label{analyticityfar}
In this section, we want to show $f(\alpha,t)$ is a real analytic function near $0$ for each fixed t, $0< t< t_3$. We first show that it is enough to prove
\begin{align}\label{anaequation}
    \frac{ic(\alpha)t}{1+ic^{'}(\alpha)\gamma t}\frac{d}{d\alpha}z(\alpha,\gamma, t)-\frac{d}{d\gamma}z(\alpha,\gamma,t)=0.
\end{align}
\begin{lemma}
If $z$ satisfies \eqref{anaequation}, then $f(x)$ can be analytically extended to $D_A=\{\alpha+iy|-\infty<\alpha<\infty, -c(\alpha)t\leq y\leq c(\alpha)t\}.$
\end{lemma}
\begin{proof}
    From the uniqueness \eqref{uniqueness}, we have $z(\alpha,0,t)=f^c(\alpha,t)$. Then
    \begin{align*}
    f^c(\alpha+iy,t)=\begin{cases}&z(\alpha,\frac{y}{c(\alpha)t},t),\ c(\alpha)\neq 0,\\
    &z(\alpha,0,t),\ c(\alpha)=0,
    \end{cases}
    \end{align*}is a extension of $f^c(\alpha,t)$ on $D_A.$ 
Moreover, when $c(\alpha)\neq 0$, we have
\[
\frac{d}{d\alpha}(z(\alpha,\frac{y}{c(\alpha)t},t))=(\partial_{\alpha}z)(\alpha,\frac{y}{c(\alpha)t},t)-(\partial_{\gamma}z)(\alpha,\frac{y}{c(\alpha)t},t)(\frac{yc'(\alpha)}{c(\alpha)^2 t}),
\]
\[
\frac{d}{dy}(z(\alpha,\frac{y}{c(\alpha)t},t))=(\partial_{\gamma}z)(\alpha,\frac{y}{c(\alpha)t},t)(\frac{1}{c(\alpha)t}).
\] 
 Then we have
\begin{align*}
    &\frac{d}{d\alpha}f^c(\alpha+iy,t)+i\frac{d}{dy}f^c(\alpha+iy,t)\\
    &=(\partial_{\alpha}z)(\alpha,\frac{y}{c(\alpha)t},t)-(\partial_{\gamma}z)(\alpha,\frac{y}{c(\alpha)t},t)(\frac{yc'(\alpha)}{c(\alpha)^2 t})+i(\partial_{\gamma}z)(\alpha,\frac{y}{c(\alpha)t},t)(\frac{1}{c(\alpha)t}).
\end{align*}
Now let $\frac{y}{c(\alpha)t}=\gamma$. Then 
\begin{align*}
     &\frac{d}{d\alpha}f^c(\alpha+iy,t)+i\frac{d}{dy}f^c(\alpha+iy,t)=\partial_{\alpha}z(\alpha,\gamma,t)-(\frac{c'(\alpha)\gamma}{c(\alpha)}-\frac{i}{c(\alpha)t})\partial_{\gamma}z(\alpha,\gamma,t)\\
     &=\partial_{\alpha}z(\alpha,\gamma,t)-(\frac{ic'(\alpha)\gamma t +1}{ic(\alpha)t})\partial_{\gamma}z(\alpha,\gamma,t)\\
     &=0.
\end{align*}
Moreover, $z(\alpha,\gamma,t)\in C^1_{\gamma}([-1,1],H^2_{\alpha}(\mathbb{T}))$. Then $\partial_{\alpha}f^c$, $\partial_{\gamma}f^c$ are continuous. Therefore we have the analyticity of $f^c$ near 0. We also have 
$f(\alpha,t)=f^c(\alpha,t)$ when $|\alpha|\leq \delta$. Then we have the result.
\end{proof}
Let 
\[
A_0(h)(\alpha,\gamma,t)=(\frac{ic(\alpha)t}{1+ic'(\alpha)\gamma t}\partial_{\alpha}-\partial_{\gamma})h(\alpha,\gamma,t).
\]
Before we prove $A_0(z)=0$,
we introduce some general lemmas.
\begin{lemma}\label{switch}
If all the derivatives are well-defined and $\partial_{\alpha}\partial_{\gamma}h(\alpha,\gamma,t)=\partial_{\gamma}\partial_{\alpha}h(\alpha,\gamma,t)$,  we have
\begin{align*}
(\frac{ic(\alpha)t}{1+ic'(\alpha)\gamma t}\partial_{\alpha}-\partial_{\gamma})\frac{\partial_{\alpha}h(\alpha,\gamma)}{1+ic'(\alpha)\gamma t}=\frac{\partial_{\alpha}}{1+ic'(\alpha)\gamma t}(\frac{ic(\alpha)t}{1+ic'(\alpha)\gamma t}\partial_{\alpha}-\partial_{\gamma})h(\alpha,\gamma),
\end{align*}
\end{lemma}

\begin{proof}
First, for the right hand side, we have
\begin{align*}
    (RHS)=\frac{\frac{ic(\alpha)t}{1+ic'(\alpha)\gamma t}\partial_{\alpha}^2h(\alpha,\gamma)}{1+ic'(\alpha)\gamma t}+\frac{ic(\alpha)t}{1+ic'(\alpha)\gamma t}(\frac{-ic''(\alpha)\gamma t}{(1+ic'(\alpha)\gamma t )^2}\partial_{\alpha}h(\alpha,\gamma))-\frac{\partial_{\alpha}\partial_{\gamma}h(\alpha,\gamma)}{1+ic'(\alpha)\gamma t}+\frac{ic'(\alpha)t\partial_{\alpha}h(\alpha,\gamma)}{(1+ic'(\alpha)\gamma t)^2}.
\end{align*}
Also
\begin{align*}
    (LHS)=[\frac{ic'(\alpha)t}{(1+ic'(\alpha)\gamma t)^2}-\frac{ic(\alpha)tic''(\alpha)\gamma t}{(1+ic'(\alpha)\gamma t)^3}]\partial_{\alpha}h(\alpha,\gamma)+\frac{ic(\alpha)t}{(1+ic'(\alpha)\gamma t)^2}\partial_{\alpha}^2h(\alpha,\gamma)-\frac{1}{1+ic'(\alpha)\gamma t}\partial_{\alpha}\partial_{\gamma}h(\alpha,\gamma).
\end{align*}
From the two equalities above, we have the result.
\end{proof}
 \begin{lemma}\label{generalalemma}
If all the derivatives are well-defined and we have
\[
\frac{d}{dt}h= \frac{ic(\alpha)\gamma}{1+ic'(\alpha)\gamma t} \partial_{\alpha}h+\tilde{T}(h),
\]
and 
$\frac{d}{d\alpha}\frac{d}{dt}h=\frac{d}{dt}\frac{d}{d\alpha}h$, $ \frac{d}{dt}\frac{d}{d\gamma}h=\frac{d}{d\gamma}\frac{d}{dt}h$, $\frac{d}{d\alpha}\frac{d}{d\gamma}h=\frac{d}{d\gamma}\frac{d}{d\alpha}h$,
then we have
\[
\frac{d}{dt}A_0(h)= \frac{ic(\alpha)\gamma}{1+ic'(\alpha)\gamma t} \partial_{\alpha}A_0(h)+A_0(\tilde{T}(h)).\]
\end{lemma}
\begin{proof}
First,
\begin{align*}
    &\frac{d}{dt}A_0(h)(\alpha,\gamma,t)\\
    &=\frac{d}{dt}(\frac{ic(\alpha )t}{1+ic'(\alpha)\gamma t})\partial_{\alpha}h+\frac{ic(\alpha)t}{1+ic'(\alpha)\gamma t}\partial_{\alpha}\frac{d}{dt}h-\partial_{\gamma}\frac{d}{dt}h\\
    &=\frac{d}{dt}(\frac{ic(\alpha )t}{1+ic'(\alpha)\gamma t})\partial_{\alpha}h+(\frac{ic(\alpha)t}{1+ic'(\alpha)\gamma t}\partial_{\alpha}-\partial_{\gamma})(\frac{ic(\alpha)\gamma}{1+ic'(\alpha)\gamma t}\partial_{\alpha}h)+A_0(\tilde{T}(h))\\
    &=\frac{d}{dt}(\frac{ic(\alpha )t}{1+ic'(\alpha)\gamma t})\partial_{\alpha}h+\frac{ic(\alpha )t}{1+ic'(\alpha)\gamma t}\partial_{\alpha}(\frac{ic(\alpha )\gamma}{1+ic'(\alpha)\gamma t})\partial_{\alpha}h+\frac{ic(\alpha )t}{1+ic'(\alpha)\gamma t}\frac{ic(\alpha )\gamma}{1+ic'(\alpha)\gamma t}\partial_{\alpha}^2 h\\
    &-\partial_{\gamma}(\frac{ic(\alpha )\gamma}{1+ic'(\alpha)\gamma t})\partial_{\alpha}h-(\frac{ic(\alpha )\gamma}{1+ic'(\alpha)\gamma t}\partial_{\gamma}\partial_{\alpha}h)+A_0(\tilde{T}(h))\\
    &=(\underbrace{\frac{ic(\alpha )t}{1+ic'(\alpha)\gamma t}\partial_{\alpha}(\frac{ic(\alpha )\gamma}{1+ic'(\alpha)\gamma t})}_{Term_1}+\underbrace{\frac{d}{dt}(\frac{ic(\alpha )t}{1+ic'(\alpha)\gamma t})-\partial_{\gamma}(\frac{ic(\alpha )\gamma}{1+ic'(\alpha)\gamma t})}_{Term_2})\partial_{\alpha}h\\
    &+\underbrace{\frac{ic(\alpha )\gamma}{1+ic'(\alpha)\gamma t}(\frac{ic(\alpha )t}{1+ic'(\alpha)\gamma t}\partial_{\alpha}-\partial_{\gamma})\partial_{\alpha}h}_{Term_{3}}+A_0(\tilde{T}(h)).
\end{align*}
Moreover, we have
\begin{align*}
 &\underbrace{\frac{d}{dt}(\frac{ic(\alpha )t}{1+ic'(\alpha)\gamma t})-\partial_{\gamma}(\frac{ic(\alpha )\gamma}{1+ic'(\alpha)\gamma t})}_{Term_2}=0,
\end{align*}
\begin{align*}
  \underbrace{\frac{ic(\alpha )t}{1+ic'(\alpha)\gamma t}\partial_{\alpha}(\frac{ic(\alpha )\gamma}{1+ic'(\alpha)\gamma t})}_{Term_1}= \frac{ic(\alpha )\gamma}{1+ic'(\alpha)\gamma t}\partial_{\alpha}(\frac{ic(\alpha )t}{1+ic'(\alpha)\gamma t}), 
\end{align*}
and 
\begin{align*}
   &\underbrace{\frac{ic(\alpha )\gamma}{1+ic'(\alpha)\gamma t}(\frac{ic(\alpha )t}{1+ic'(\alpha)\gamma t}\partial_{\alpha}-\partial_{\gamma})\partial_{\alpha}h}_{Term_3}\\
   &= \frac{ic(\alpha )\gamma}{1+ic'(\alpha)\gamma t}\partial_{\alpha}((\frac{ic(\alpha )t}{1+ic'(\alpha)\gamma t}\partial_{\alpha}-\partial_{\gamma})h)-\frac{ic(\alpha )\gamma}{1+ic'(\alpha)\gamma t}\partial_{\alpha}(\frac{ic(\alpha )t}{1+ic'(\alpha)\gamma t})\partial_{\alpha}h.
\end{align*}
Therefore 
\[
\frac{d}{dt}A_0(h)(\alpha,\gamma,t)= \frac{ic(\alpha )\gamma}{1+ic'(\alpha)\gamma t}\partial_{\alpha}((\frac{ic(\alpha )t}{1+ic'(\alpha)\gamma t}\partial_{\alpha}-\partial_{\gamma})h)+A_0(\tilde{T}(h))=\frac{ic(\alpha )\gamma}{1+ic'(\alpha)\gamma t}\partial_{\alpha}A_0(h)+A_0(\tilde{T}(h)).
\]
\end{proof}
\begin{lemma}\label{forM1}
Let $\tilde{K}$ be meromorphic. $\partial_{\alpha}X(\alpha,\gamma)$, $\partial_{\gamma}X(\alpha,\gamma)$ are well-defined and in $C_{\alpha}^{0}[-\pi,\pi]$, $\partial_{\alpha}h(\alpha,\gamma)$ and $\partial_{\gamma}h(\alpha,\gamma)$ are well-defined vector functions with components in $C_{\alpha}^{0}[-\pi,\pi]$. If for fixed $\alpha$, there is no singular point in the integrals below and $c(\pi)=c(-\pi)=0$, $c(\alpha)\in W^{2,\infty}$, then we have
\begin{align*}
    &A_0(\int_{-\pi}^{\pi}\tilde{K}(h(\alpha,\gamma)-h(\beta,\gamma))X(\beta,\gamma)(1+ic'(\beta)\gamma t)d\beta)=
    \\&\int_{-\pi}^{\pi}\nabla \tilde{K}(h(\alpha,\gamma)-h(\beta,\gamma))\cdot(A_0(h)(\alpha,\gamma)-A_0(h)(\beta,\gamma))X(\beta,\gamma)(1+ic'(\beta)\gamma t)d\beta\\
    &+\int_{-\pi}^{\pi}\tilde{K}(h(\alpha,\gamma)-h(\beta,\gamma))A_0(X)(\beta,\gamma)(1+ic'(\beta)\gamma t)d\beta\\
    &=D_{h}(\int_{-\pi}^{\pi}\tilde{K}(h(\alpha,\gamma)-h(\beta,\gamma))X(\beta,\gamma)(1+ic'(\beta)\gamma t)d\beta)[A_0(h)]\\
    &+\int_{-\pi}^{\pi}\tilde{K}(h(\alpha,\gamma)-h(\beta,\gamma))A_0(X)(\beta,\gamma)(1+ic'(\beta)\gamma t)d\beta.
\end{align*}
Here $D_h$ is the Gateaux derivative.
\end{lemma}
\begin{proof}
We have
\begin{align*}
 &A_0(\int_{\alpha-\pi}^{\alpha+\pi}\tilde{K}(h(\alpha,\gamma)-h(\alpha-\beta,\gamma))X(\alpha-\beta,\gamma)(1+ic'(\alpha-\beta)\gamma t)d\beta)=
    \\
    &\underbrace{\frac{ic(\alpha)t}{1+ic'(\alpha)\gamma t}(\tilde{K}(h(\alpha,\gamma)-h(-\pi,\gamma))X(-\pi,\gamma)(1+ic'(-\pi)\gamma t)-\tilde{K}(h(\alpha,\gamma)-h(\pi,\gamma))X(\pi,\gamma)(1+ic'(\pi)\gamma t))}_{Term_1}\\
    &+\underbrace{\int_{\alpha-\pi}^{\alpha+\pi}\nabla \tilde{K}(h(\alpha,\gamma)-h(\alpha-\beta,\gamma))\cdot((\frac{ic(\alpha)t}{1+ic'(\alpha)\gamma t}\partial_{\alpha}-\partial_{\gamma})h(\alpha,\gamma)-(\frac{ic(\alpha)t}{1+ic'(\alpha)\gamma t}\partial_{\alpha}-\partial_{\gamma})h(\alpha-\beta,\gamma,t))}_{Term_2}\\
    &\underbrace{X(\alpha-\beta,\gamma)(1+ic'(\alpha-\beta)\gamma t)d\beta}_{Term_2}\\
    &+\underbrace{\int_{\alpha-\pi}^{\alpha+\pi}\tilde{K}(h(\alpha,\gamma)-h(\alpha-\beta,\gamma))(\frac{ic(\alpha)t}{1+ic'(\alpha)\gamma t}\partial_{\alpha}-\partial_{\gamma})X(\alpha-\beta,\gamma)(1+ic'(\alpha-\beta)\gamma t)d\beta}_{Term_3}\\
    &+\underbrace{\int_{\alpha-\pi}^{\alpha+\pi}\tilde{K}(h(\alpha,\gamma)-h(\alpha-\beta,\gamma))X(\alpha-\beta,\gamma)(\frac{ic(\alpha)t}{1+ic'(\alpha)\gamma t}ic''(\alpha-\beta)\gamma t-ic'(\alpha-\beta) t)d\beta}_{Term_4}\\
    &=Term 1 +\\     
    &+\underbrace{\int_{\alpha-\pi}^{\alpha+\pi}\nabla \tilde{K}(h(\alpha,\gamma)-h(\alpha-\beta,\gamma))\cdot((\frac{ic(\alpha)t}{1+ic'(\alpha)\gamma t}\partial_{\alpha}-\partial_{\gamma})h(\alpha,\gamma)-(\frac{ic(\alpha-\beta)t}{1+ic'(\alpha-\beta)\gamma t}\partial_{\alpha}-\partial_{\gamma})h(\alpha-\beta,\gamma,t))}_{Term_{2,1}}\\
    &\underbrace{X(\alpha-\beta,\gamma)(1+ic'(\alpha-\beta)\gamma t)d\beta}_{Term_{2,1}}\\
    &+\underbrace{\int_{\alpha-\pi}^{\alpha+\pi}\nabla \tilde{K}(h(\alpha,\gamma)-h(\alpha-\beta,\gamma))\cdot((\frac{ic(\alpha-\beta)t}{1+ic'(\alpha-\beta)\gamma t}\partial_{\alpha}-\frac{ic(\alpha)t}{1+ic'(\alpha)\gamma t}\partial_{\alpha})h(\alpha-\beta,\gamma,t))}_{Term_{2,2}}\\
    &\underbrace{X(\alpha-\beta,\gamma)(1+ic'(\alpha-\beta)\gamma t)d\beta}_{Term_{2,2}}\\
    &+\underbrace{\int_{\alpha-\pi}^{\alpha+\pi}\tilde{K}(h(\alpha,\gamma)-h(\alpha-\beta,\gamma))(\frac{ic(\alpha-\beta)t}{1+ic'(\alpha-\beta)\gamma t}\partial_{\alpha}-\partial_{\gamma})X(\alpha-\beta,\gamma)(1+ic'(\alpha-\beta)\gamma t)d\beta}_{Term_{3,1}}\\
    &+\underbrace{\int_{\alpha-\pi}^{\alpha+\pi}\tilde{K}(h(\alpha,\gamma)-h(\alpha-\beta,\gamma))(\frac{ic(\alpha)t}{1+ic'(\alpha)\gamma t}\partial_{\alpha}-\frac{ic(\alpha-\beta)t}{1+ic'(\alpha-\beta)\gamma t}\partial_{\alpha})X(\alpha-\beta,\gamma)(1+ic'(\alpha-\beta)\gamma t)d\beta}_{Term_{3,2}}\\
     &+\underbrace{\int_{\alpha-\pi}^{\alpha+\pi}\tilde{K}(h(\alpha,\gamma)-h(\alpha-\beta,\gamma))X(\alpha-\beta,\gamma)(\frac{ic(\alpha)t}{1+ic'(\alpha)\gamma t}ic''(\alpha-\beta)\gamma t-ic'(\alpha-\beta) t)d\beta}_{Term_4}\\
    &=Term_1+Term_2+Term_{2,2}+Term_{3,1}+Term_{3,2}+Term_4.
\end{align*}
Moreover,
\begin{align*}
    &Term_{2,2}+Term_{3,2}+Term_4=\\
    &=\int_{\alpha-\pi}^{\alpha+\pi}\frac{d}{d\beta}(\tilde{K}(h(\alpha,\gamma)-h(\alpha-\beta,\gamma))X(\alpha-\beta,\gamma)(ic(\alpha-\beta)t-\frac{ic(\alpha)t}{1+ic'(\alpha)\gamma t}(1+ic'(\alpha-\beta)\gamma t)))d\beta\\
    &=\tilde{K}(h(\alpha,\gamma)-h(-\pi,\gamma))X(-\pi,\gamma)(ic(-\pi)t-\frac{ic(\alpha)t}{1+ic'(\alpha)\gamma t}(1+ic'(-\pi)\gamma t))\\
    &-\tilde{K}(h(\alpha,\gamma)-h(\pi,\gamma))X(\pi,\gamma)(ic(\pi)t-\frac{ic(\alpha)t}{1+ic'(\alpha)\gamma t}(1+ic'(\pi)\gamma t)).
\end{align*}
We use the condition that $c(-\pi)=c(\pi)=0$ and we could get the $Term_1+Term_{2,2}+Term_{3,2}+Term_4=0$. Then we have the result.
\end{proof}
Now we use lemma \ref{switch}, \ref{generalalemma}, \ref{forM1}, to show the result. 
From \eqref{equationz} and \eqref{kequationnew}, we have
\begin{align*}
&\frac{d z_{\mu}(\alpha, \gamma,t)}{dt}=\frac{ic(\alpha)\gamma}{1+ic^{'}(\alpha)\gamma t}\partial_{\alpha}z_{\mu}(\alpha,\gamma,t)+\\\nonumber
&+\lambda(\alpha_{\gamma}^{t})\int_{-\pi}^{\pi}K(z(\alpha,\gamma,t)-z(\beta,\gamma,t),\tilde{f}(\alpha_{\gamma}^t,t)-\tilde{f}(\beta_{\gamma}^t,t)) (\frac{\partial_{\alpha}z_{\mu}(\alpha,\gamma,t)}{1+ic'(\alpha)\gamma t}-\frac{\partial_{\beta}z_{\mu}(\beta,\gamma,t)}{1+ic'(\beta)\gamma t})(1+ic'(\beta)\gamma t)d\beta\\\nonumber
&+\lambda(\alpha_{\gamma}^t)\int_{-\pi}^{\pi}K(z(\alpha,\gamma,t)-z(\beta,\gamma,t),\tilde{f}(\alpha_{\gamma}^t,t)-\tilde{f}(\beta_{\gamma}^t,t))((\partial_{\alpha}\tilde{f_{\mu}})(\alpha_{\gamma}^t,t)-(\partial_{\beta}\tilde{f_{\mu}})(\beta_{\gamma}^t,t))(1+ic'(\beta)\gamma t)d\beta.
\end{align*}
Since we have $A_0(\lambda(\alpha_{\gamma}^{t}))=0$, from lemma \ref{generalalemma}, we get
\begin{align*}
     &\frac{dA_0(z_{\mu})}{dt}=\frac{ic(\alpha)\gamma}{1+ic^{'}(\alpha)\gamma t}\partial_{\alpha}A_0(z_{\mu})(\alpha,\gamma,t)\\
   &+\lambda(\alpha_{\gamma}^{t})A_0(\int_{-\pi}^{\pi}K(z(\alpha,\gamma,t)-z(\beta,\gamma,t),\tilde{f}(\alpha_{\gamma}^t,t)-\tilde{f}(\beta_{\gamma}^t,t))\\
   &\nonumber(\frac{\partial_{\alpha}z_{\mu}(\alpha,\gamma,t)}{1+ic'(\alpha)\gamma t}-\frac{\partial_{\beta}z_{\mu}(\beta,\gamma,t)}{1+ic'(\beta)\gamma t}+(\partial_{\alpha}\tilde{f_{\mu}})(\alpha_{\gamma}^t,t)-(\partial_{\beta}\tilde{f_{\mu}})(\beta_{\gamma}^t,t))(1+ic'(\beta)\gamma t)d\beta).
\end{align*}

 Let 

\[h(\alpha,\gamma)=(z(\alpha,\gamma,t),\tilde{f}(\alpha_{\gamma}^t,t), \frac{\partial_{\alpha}z_{\mu}(\alpha,\gamma,t)}{1+ic'(\alpha)\gamma t}+(\partial_{\alpha}\tilde{f_{\mu}})(\alpha_{\gamma}^t,t)),
\]
\[
\tilde{K}(h(\alpha,\gamma)-h(\beta,\gamma))=K(h_1(\alpha,\gamma)-h_1(\beta,\gamma),h_2(\alpha,\gamma)-h_2(\beta,\gamma))(h_3(\alpha,\gamma)-h_3(\beta,\gamma)),
\] 
and 
\[
X(\alpha,\gamma)=1.
\]
then by lemma \ref{forM1} and lemma \ref{switch} and $A_0((\partial_{\alpha}\tilde{f_{\mu}})(\alpha_{\gamma}^t,t))=0$, and $A_0(\tilde{f_{\mu}}(\alpha_{\gamma}^t,t))=0$, we have
\begin{align*}
    &\frac{dA_0(z_{\mu})}{dt}=\frac{ic(\alpha)\gamma}{1+ic^{'}(\alpha)\gamma t}\partial_{\alpha}A_0(z_{\mu})(\alpha,\gamma,t)\\
    &+\lambda(\alpha_{\gamma}^{t})\int_{-\pi}^{\pi}K(z(\alpha,\gamma,t)-z(\beta,\gamma,t),\tilde{f}(\alpha_{\gamma}^t,t)-\tilde{f}(\beta_{\gamma}^t,t)) (\frac{\partial_{\alpha}A_0(z_{\mu})(\alpha,\gamma,t)}{1+ic'(\alpha)\gamma t}-\frac{\partial_{\beta}A_0(z_{\mu})(\beta,\gamma,t)}{1+ic'(\beta)\gamma t})\\
    &\cdot(1+ic'(\beta)\gamma t)d\beta\\
    &+\lambda(\alpha_{\gamma}^{t})\int_{-\pi}^{\pi}\nabla_1 K(z(\alpha,\gamma,t)-z(\beta,\gamma,t),\tilde{f}(\alpha_{\gamma}^t,t)-\tilde{f}(\beta_{\gamma}^t,t))\cdot(A_0(z)(\alpha,\gamma,t)-A_0(z)(\beta,\gamma,t))\\
    &\cdot(\frac{\partial_{\alpha}z_{\mu}(\alpha,\gamma,t)}{1+ic'(\alpha)\gamma t}-\frac{\partial_{\beta}z_{\mu}(\beta,\gamma,t)}{1+ic'(\beta)\gamma t}+(\partial_{\alpha}\tilde{f_{\mu}})(\alpha_{\gamma}^t,t)-(\partial_{\beta}\tilde{f_{\mu}})(\beta_{\gamma}^t,t))(1+ic'(\beta)\gamma t)d\beta\\
    &=Term_1+Term_2+Term_3.
\end{align*}
Since 
\[
\|\nabla_1 K(z(\alpha,\gamma,t)-z(\beta,\gamma,t)+\tilde{f}(\alpha_{\gamma}^t,t)-\tilde{f}(\beta_{\gamma}^t,t))(\alpha-\beta)^2\|_{C^1([-2\delta,2\delta]\times[-\pi,\pi])}\lesssim1,
\]
\[
\|\frac{\partial_{\alpha}z_{\mu}(\alpha,\gamma,t)}{1+ic'(\alpha)\gamma t}+(\partial_{\alpha}\tilde{f_{\mu}})(\alpha_{\gamma}^t,t)\|_{C^2([-2\delta,2\delta]\times[-\pi,\pi])}\lesssim 1,
\]
 by lemma \ref{goodterm2}, we have
\begin{align*}
    \|Term_3\|_{L^2(\mathbb{T})}\lesssim\|A_0(z)\|_{L^2(\mathbb{T})}.
\end{align*}
Then by corollary \ref{maintermcorollary}, we have
\begin{align*}
    \frac{d}{dt}\|A_0(z)\|_{L^2(\mathbb{T})}^2\lesssim\|A_0(z)\|_{L^2(\mathbb{T})}^2.
\end{align*}
Moreover when $t=0$, $A_0(z)=-\partial_{\gamma}z(\alpha,\gamma,0)=-\partial_{\gamma}f(\alpha,0)=0$.
Therefore $A_0(z)=0$, when $t\leq t_3$.
\section{Using the energy estimate to show the analyticity}\label{analyticitystationary}
Following the similar idea as previous sections, we introduce a way to study the analyticity of the solution to some differential equations, which is to our best knowledge a new method.
\begin{theorem}\label{analyticity0}
Let $T(f)$ be an operator satisfying the following conditions. We assume there exists $\epsilon>0$, $k\geq  1$,  $f_0\in H^{k}(\mathbb{T})$, such that when $\|f-f_0\|_{H^{k}}\lesssim \epsilon$,
\begin{itemize}
    \item [(a)](Boundedness) $T(f):H^{k}(\mathbb{T})\to H^{k}(\mathbb{T})$ with $\|T(f)\|_{H^{k}(\mathbb{T})}\lesssim 1$,
    \item [(b)](Existence and boundedness of the Fréchet derivative )$\|D_{f}(T(f))[h]\|_{H^{k}(\mathbb{T})}\lesssim \|h\|_{H^k(\mathbb{T})},$
    \item  [(c)]($L^2$ boundedness of the Fréchet derivative) $\|D_{f}(T(f))[h]\|_{L^{2}(\mathbb{T})}\lesssim \|h\|_{L^2(\mathbb{T})}$,
     \item[(d)]$\frac{d}{dx}T(f)=D_{f}(T(f))[\frac{df}{dx}]$,
     \item[(e)] $iD_f(T(f))[h]=D_f(T(f))[ih]$.
\end{itemize}
Here $\mathbb{T}$ is the torus of length $2\pi$. If $f_0(x)$ also satisfies the equation
\begin{equation}\label{original equation1}
    f_0^{'}(x)=T(f_0),
\end{equation}
$f_{0}(x)$ must be real analytic.
\end{theorem}
\begin{proof}
 First, we assume $f_0$ to be an analytic function with analytic continuation $f(x,t)$. Then through the Cauchy-Riemann equation and \eqref{original equation1} , we have
  \begin{equation}\label{extended equation1}
\begin{cases}
\frac{d}{dt}f(x,t)=iT(f(x,t)),\\
f(x,0)=f_0(x).
\end{cases}
\end{equation}
Our goal is to show this solution $f(x,t)$ does exist and is analytic.

 Through (a), (b), we have 
 \[
 \|T(f)-T(g)\|_{H^{k}(\mathbb{T})}\lesssim \|f-g\|_{H^k(\mathbb{T})}.
 \]
 We can use the Picard theorem to show there is a solution satisfying
\[
f(x,t)-f(x,0)=\int_{0}^{t}iT(f(x,\tau))d\tau.
\]
with $|t|<t_0$ for some $t_0>0$. Moreover 
\begin{equation}\label{fspace}
 f(x,t)\in W^{2,\infty}((-t_0, t_0),H^k(\mathbb{T})).   
\end{equation}
By \eqref{fspace}, we have
\[
\lim_{\Delta t \to 0}\|\frac{f(x,t)-f(x,t+\Delta t)}{\Delta t}-\frac{d}{dt}f(x,t)\|_{H^{k}(\mathbb{T})}=0.
\]
Hence 
\[
\lim_{\Delta t \to 0}\|\frac{\frac{d}{dx}f(x,t)-\frac{d}{dx}f(x,t+\Delta t)}{\Delta t}-\frac{d}{dx}\frac{d}{dt}f(x,t)\|_{H^{k-1}(\mathbb{T})}=0.
\]
Therefore we have
\begin{equation}\label{fspace2}
\frac{d}{dt}\frac{d}{dx}f(x,t)=\frac{d}{dx}\frac{d}{dt}f(x,t)\in C^{0}((-t_0, t_0),H^{k-1}(\mathbb{T}))\subset C^{0}((-t_0, t_0),L^2(\mathbb{T})),
\end{equation}
and
\begin{equation}\label{fspace3}
\frac{d}{dx}f(x,t)+i\frac{d}{dt}f(x,t)\in W^{1,\infty}((-t_0,t_0), L^2(\mathbb{T})).
\end{equation}
Then we can control $\|\frac{d}{dx}f(x,t)+i\frac{d}{dt}f(x,t)\|_{L^2(\mathbb{T})}$. We have
\begin{align}\label{L2estimateT}
    &|\frac{d}{dt}\int_{-\pi}^{\pi}|(\frac{d}{dx}+i\frac{d}{dt})f(x,t)|^2dx|\\\nonumber
    &=|2Re\int_{-\pi}^{\pi}(\frac{d}{dx}+i\frac{d}{dt})f(x,t)\overline{\frac{d}{dt}(\frac{d}{dx}+i\frac{d}{dt})f(x,t)}dx|\\\nonumber
    &=|2Re\int_{-\pi}^{\pi}(\frac{d}{dx}+i\frac{d}{dt})f(x,t)\overline{(\frac{d}{dx}+i\frac{d}{dt})\frac{d}{dt}f(x,t)}dx|\\\nonumber
    &=|2Re\int_{-\pi}^{\pi}(\frac{d}{dx}+i\frac{d}{dt})f(x,t)\overline{(\frac{d}{dx}+i\frac{d}{dt})iT(f(x,t))}dx|\\\nonumber
    &=|2Re\int_{-\pi}^{\pi}(\frac{d}{dx}+i\frac{d}{dt})f(x,t)\overline{iD_{f}T(f(x,t))[(\frac{d}{dx}+i\frac{d}{dt})[f(x,t)]}dx|\\\nonumber
    &\lesssim\int_{-\pi}^{\pi}|(\frac{d}{dx}+i\frac{d}{dt})f(x,t)|^2dx.
\end{align}
Here the first equality follows from \eqref{fspace3}, the second from \eqref{fspace2}.

Through \eqref{original equation1}, \eqref{extended equation1}, we have
\[
\|\frac{d}{dx}f(x,t)+i\frac{d}{dt}f(x,t)\|_{L^2(\mathbb{T})}|_{t=0}=0.
\]
Moreover, from\eqref{fspace3}, $\|\frac{d}{dx}f(x,t)+i\frac{d}{dt}f(x,t)\|_{L^2(\mathbb{T})}^2\in W^{1,\infty}(-t_0,t_0).$ Then 
\[
\|\frac{d}{dx}f(x,t)+i\frac{d}{dt}f(x,t)\|_{L^2(\mathbb{T})}^2=\int_{0}^{t}\frac{d}{d\tau}\|\frac{d}{dx}f(x,\tau)+i\frac{d}{d\tau}f(x,\tau)\|_{L^2(\mathbb{T})}^2d\tau.
\]
Hence we can use the Gronwall inequality and get 
\begin{equation}\label{fL2}
\|\frac{d}{dx}f(x,t)+i\frac{d}{dt}f(x,t)\|_{L^2(\mathbb{T})}=0.
\end{equation}
Moreover, since $k\geq 1$, from \eqref{fspace}, we have $\frac{d}{dt}f(x,t)\in W^{1,\infty}((-t_0,t_0),H^{1}(\mathbb{T}))$. Then $\frac{d}{dt}f(x,t)$ is continuous in $x$ and $t$. 
Therefore $\partial_{x}f(x,t)$ is continuous in $x$ and $t$.

Then by the \eqref{fL2} and \eqref{extended equation1},  we have the analyticity.
\end{proof}
\section{Appendix}
\begin{lemma}\label{goodterm0}
For $G(\alpha,\beta)\in C^1([-2\delta,2\delta]\times[-\pi,\pi])$, we have
\begin{align*}
&\|p.v.\int_{-\pi}^{\pi}\frac{G(\alpha,\beta)}{(\alpha-\beta)}d\beta\|_{L^\infty[-2\delta,2\delta]}\lesssim \|G(\alpha,\beta)\|_{C^1([-2\delta,2\delta]\times[-\pi,\pi])}.
\end{align*}
\begin{proof}
We have
\begin{align*}
&|p.v.\int_{-\pi}^{\pi}\frac{G(\alpha,\beta)}{(\alpha-\beta)}d\beta|\\
&=|\int_{-\pi}^{\pi}\frac{G(\alpha,\beta)-G(\alpha,\alpha)}{(\alpha-\beta)}d\beta|+|p.v.\int_{-\pi}^{\pi}\frac{1}{\alpha-\beta}d\beta G(\alpha,\alpha)|\lesssim
\|G(\alpha,\beta)\|_{C^1([-2\delta,2\delta]\times[-\pi,\pi])}.
\end{align*}
\end{proof}
\end{lemma}
\begin{lemma}\label{goodterm1}
For $g(\alpha)\in L^2[-\pi,\pi]$, $G(\alpha,\beta)\in C^1([-2\delta,2\delta]\times[-\pi,\pi])$, we have
\begin{align*}
&\|\int_{-\pi}^{\pi}G(\alpha,\beta)\frac{g(\alpha)-g(\beta)}{(\alpha-\beta)}d\beta\|_{L^2[-2\delta,2\delta]}\lesssim \|G(\alpha,\beta)\|_{C^1([-2\delta,2\delta]\times[-\pi,\pi])}\|g\|_{L^2[-\pi,\pi]}.
\end{align*}
\end{lemma}
\begin{proof}
We have
\begin{align*}
 &\int_{-\pi}^{\pi}G(\alpha,\beta)\frac{g(\alpha)-g(\beta)}{(\alpha-\beta)}d\beta\\
 &=\int_{-\pi}^{\pi}\frac{G(\alpha,\beta)-G(\alpha,\alpha)}{\alpha-\beta}(g(\alpha)-g(\beta))d\beta+G(\alpha,\alpha)p.v.\int_{-\pi}^{\pi}\frac{g(\alpha)-g(\beta)}{\alpha-\beta}d\beta\\
&=Term_{1}+Term_{2}.
\end{align*}
Here
\begin{align*}
    \|Term_1\|_{L^2[-2\delta,2\delta]}\lesssim\|G(\alpha,\beta)\|_{C^1([-2\delta,2\delta]\times[-\pi,\pi])}\|g(\alpha)\|_{L^2[-\pi,\pi]},
\end{align*}
and
\begin{align*}
    \|Term_2\|_{L^2[-2\delta,2\delta]}\lesssim\|G(\alpha,\alpha)\|_{C^0([-2\delta,2\delta])}\|g(\alpha)\|_{L^2[-\pi,\pi]}.
\end{align*}
\end{proof}
\begin{corollary}\label{goodterm2}
 For $g(\alpha)\in L^2[-\pi,\pi]$, $h(\alpha)\in C^2[-\pi,\pi]$, $G(\alpha,\beta)\in C^1([-2\delta,2\delta]\times[-\pi,\pi])$, we have
\begin{align*}
&\|\int_{-\pi}^{\pi}G(\alpha,\beta)\frac{(g(\alpha)-g(\beta))(h(\alpha)-h(\beta))}{(\alpha-\beta)^2}d\beta\|_{L^2[-2\delta,2\delta]}\lesssim \|G(\alpha,\beta)\|_{C^1([-2\delta,2\delta]\times[-\pi,\pi])}\|g\|_{L^2[-2\delta,2\delta]}\|h\|_{C^2[-\pi,\pi]}.
 \end{align*}
\end{corollary}
\begin{proof}
We can use lemma \ref{goodterm1} and let $\tilde{G}(\alpha,\beta)=G(\alpha,\beta)\frac{h(\alpha)-h(\beta)}{(\alpha-\beta)}.$
\end{proof}
\begin{lemma}\label{goodterm3}
For $g(\alpha)\in H^1[-\pi,\pi]$, $G(\alpha,\beta)\in C^0([-2\delta,2\delta]\times[-\pi,\pi])$, we have
\begin{align*}
&\|\int_{-\pi}^{\pi}G(\alpha,\beta)\frac{g(\alpha)-g(\beta)}{(\alpha-\beta)}d\beta\|_{L^2[-2\delta,2\delta]}\lesssim \|G(\alpha,\beta)\|_{C^0([-2\delta,2\delta]\times[-\pi,\pi])}\|g\|_{H^1[-\pi,\pi]}.
\end{align*}
\end{lemma}
\begin{proof}
We have
\begin{align*}
    &|\int_{-\pi}^{\pi}G(\alpha,\beta)\frac{g(\alpha)-g(\beta)}{(\alpha-\beta)}d\beta|\\
    &\leq \|G(\alpha,\beta)\|_{C^0([-2\delta,2\delta]\times[-\pi,\pi])}\int_{-\pi}^{\pi}\frac{|g(\alpha)-g(\beta)|}{(\alpha-\beta)}d\beta\\
    &\leq \|G(\alpha,\beta)\|_{C^0([-2\delta,2\delta]\times[-\pi,\pi])}\int_{-\pi}^{\pi}\int_{0}^1|g'(\tau(\alpha)+(1-\tau)(\beta))|d\tau d\beta\\
    &\lesssim \|G(\alpha,\beta)\|_{C^0([-2\delta,2\delta]\times[-\pi,\pi])}\|g\|_{H^1[-\pi,\pi]}.
\end{align*}
\end{proof}
\begin{lemma}\label{goodterm4}
For $g(\alpha)\in H^k(\mathbb{T})$, $G(\alpha,\beta)\in C^1([-2\delta,2\delta]\times\mathbb{T})\cap C^k([-2\delta,2\delta]\times\mathbb{T})$, for $k\geq 0$, we have
\begin{align*}
    \|\int_{-\pi}^{\pi}G(\alpha,\beta)\frac{g(\alpha)-g(\beta)}{\alpha-\beta}d\beta\|_{H^{k}[-2\delta,2\delta]}\lesssim(\|G(\alpha,\beta)\|_{C^1([-2\delta,2\delta]\times\mathbb{T})}+\|G(\alpha,\beta)\|_{C^k([-2\delta,2\delta]\times\mathbb{T})})\|g\|_{H^k(\mathbb{T})}.
    \end{align*}
\end{lemma}
\begin{proof}
From lemma \ref{goodterm1}, when $k=0$, we have the $L^2$ norm.
Moreover, when $k\geq 1$, we have
\begin{align*}
    &\partial_{\alpha}^{k}\int_{-\pi}^{\pi}G(\alpha,\beta)\frac{g(\alpha)-g(\beta)}{\alpha-\beta}d\beta\\
    &=\partial_{\alpha}^{k}\int_{-\pi}^{\pi}G(\alpha,\alpha-\beta)\frac{g(\alpha)-g(\alpha-\beta)}{\beta}d\beta\\
    &=\sum_{0\leq j \leq k} C_j \int_{-\pi}^{\pi}\partial_{\alpha}^j G(\alpha,\alpha-\beta)\frac{\partial_{\alpha}^{k-j}g(\alpha)-\partial_{\alpha}^{k-j}g(\alpha-\beta)}{\beta}d\beta.
\end{align*}
For $j\leq k-1$, we could use lemma \ref{goodterm1} to get the estimate.
For $j= k$, we could use lemma \ref{goodterm3} to get the estimate.
\end{proof}
\begin{lemma}\label{goodterm5}
For $k\geq 2$, $g(\alpha)\in H^{k}(\mathbb{T})$, $G(\alpha,\beta)\in  C^k([-2\delta,2\delta]\times\mathbb{T})$, we have
\begin{align*}
    &\|\partial_{\alpha}^{k}(\int_{-\pi}^{\pi}G(\alpha,\beta)\frac{g(\alpha)-g(\beta)}{\alpha-\beta}d\beta)-\int_{-\pi}^{\pi}G(\alpha,\beta)\frac{\partial_{\alpha}^{k}g(\alpha)-\partial_{\beta}^{k}g(\beta)}{\alpha-\beta}d\beta\|_{L^{2}[-2\delta,2\delta]}\\
    &\lesssim\|G(\alpha,\beta)\|_{C^k([-2\delta,2\delta]\times\mathbb{T})}\|g\|_{H^{k-1}(\mathbb{T})}
\end{align*}
\end{lemma}
\begin{proof}
 We have
\begin{align*}
   &\partial_{\alpha}^{k}(\int_{-\pi}^{\pi}G(\alpha,\beta)\frac{g(\alpha)-g(\beta)}{\alpha-\beta}d\beta)-\int_{-\pi}^{\pi}G(\alpha,\beta)\frac{\partial_{\alpha}^{k}g(\alpha)-\partial_{\beta}^{k}g(\beta)}{\alpha-\beta}d\beta\\
   &=\sum_{1\leq j\leq k} C_j\int_{-\pi}^{\pi}\partial_{\alpha}^{j}(G(\alpha,\alpha-\beta))\frac{\partial_{\alpha}^{k-j}g(\alpha)-\partial_{\beta}^{k-j}g(\alpha-\beta)}{\beta}d\beta.
\end{align*}
Then for $j\leq k-1$. we could use lemma \ref{goodterm1} to  get the  estimate.
Then for $j= k$. we could use lemma \ref{goodterm3} to  get the  estimate.
\end{proof} 
\begin{lemma}\label{goodterm6}
For $g(\alpha)\in H^3(\mathbb{T})$, $h(\alpha)\in H^3(\mathbb{T})$, $G(\alpha,\beta)\in C^3([-2\delta,2\delta]\times\mathbb{T})$, we have
\begin{align*}
    &\|\partial_{\alpha}^{3}(\int_{-\pi}^{\pi}G(\alpha,\beta)\frac{g(\alpha)-g(\beta)}{\alpha-\beta}\frac{h(\alpha)-h(\beta)}{\alpha-\beta}d\beta)\|_{L^2[-2\delta,2\delta]}\\
    &\lesssim\|G(\alpha,\beta)\|_{C^3([-2\delta,2\delta]\times\mathbb{T})}\|g\|_{H^{3}(\mathbb{T})}\|h\|_{H^{3}(\mathbb{T})}.
\end{align*}
\end{lemma}
\begin{proof}
We take the derivative and have
\begin{align*}
    &\partial_{\alpha}^{3}(\int_{-\pi}^{\pi}G(\alpha,\alpha-\beta)\frac{g(\alpha)-g(\alpha-\beta)}{\beta}\frac{h(\alpha)-h(\alpha-\beta)}{\beta}d\beta)\\
    &=\sum_{j_1+j_2+j_3=3}\int_{-\pi}^{\pi}\partial_{\alpha}^{j_1}(G(\alpha,\alpha-\beta))\frac{\partial_{\alpha}^{j_2}g(\alpha)-\partial_{\alpha}^{j_2}g(\alpha-\beta)}{\beta}\frac{\partial_{\alpha}^{j_3}h(\alpha)-\partial_{\alpha}^{j_3}h(\alpha-\beta)}{\beta}d\beta.
\end{align*}
If $j_2\leq 1$, $j_3\leq 1$, then 
\begin{align*}
&\|\int_{-\pi}^{\pi}\partial_{\alpha}^{j_1}(G(\alpha,\alpha-\beta))\frac{\partial_{\alpha}^{j_2}g(\alpha)-\partial_{\alpha}^{j_2}g(\alpha-\beta)}{\beta}\frac{\partial_{\alpha}^{j_3}h(\alpha)-\partial_{\alpha}^{j_3}h(\alpha-\beta)}{\beta}d\beta\|_{L^2[-2\delta,2\delta]}\\
&\lesssim\|\partial_{\alpha}^{j_1}(G(\alpha,\alpha-\beta))\|_{C^0([-2\delta,2\delta]\times\mathbb{T})}\|g\|_{C^{2}(\mathbb{T})}\|h\|_{C^{2}(\mathbb{T})}\\
&\lesssim\|G(\alpha,\beta)\|_{C^3([-2\delta,2\delta]\times\mathbb{T})}\|g\|_{H^{3}(\mathbb{T})}\|h\|_{H^{3}(\mathbb{T})}.
\end{align*}
If $j_2=3$, then $j_1=j_3=0$, by lemma \ref{goodterm2}, we have
\begin{align*}
&\|\int_{-\pi}^{\pi}\partial_{\alpha}^{j_1}(G(\alpha,\alpha-\beta))\frac{\partial_{\alpha}^{j_2}g(\alpha)-\partial_{\alpha}^{j_2}g(\alpha-\beta)}{\beta}\frac{\partial_{\alpha}^{j_3}h(\alpha)-\partial_{\alpha}^{j_3}h(\alpha-\beta)}{\beta}d\beta\|_{L^2[-2\delta,2\delta]}\\
&\lesssim\|(G(\alpha,\alpha-\beta))\|_{C^1([-2\delta,2\delta]\times\mathbb{T})}\|g\|_{H^{3}(\mathbb{T})}\|h\|_{C^{2}(\mathbb{T})}\\
&\lesssim\|G(\alpha,\beta)\|_{C^3([-2\delta,2\delta]\times\mathbb{T})}\|g\|_{H^{3}(\mathbb{T})}\|h\|_{H^{3}(\mathbb{T})}.
\end{align*}
If $j_2=2$, then $j_1=1$ or $j_3=1$, by lemma \ref{goodterm3}, we have
\begin{align*}
&\|\int_{-\pi}^{\pi}\partial_{\alpha}^{j_1}(G(\alpha,\alpha-\beta))\frac{\partial_{\alpha}^{j_2}g(\alpha)-\partial_{\alpha}^{j_2}g(\alpha-\beta)}{\beta}\frac{\partial_{\alpha}^{j_3}h(\alpha)-\partial_{\alpha}^{j_3}h(\alpha-\beta)}{\beta}d\beta\|_{L^2[-2\delta,2\delta]}\\
&\lesssim\|\partial_{\alpha}^{j_1}(G(\alpha,\alpha-\beta))\|_{C^0([-2\delta,2\delta]\times\mathbb{T})}\|g\|_{H^{3}(\mathbb{T})}\|\partial_{\alpha}^{j_3}h\|_{C^{1}(\mathbb{T})}\\
&\lesssim\|G(\alpha,\beta)\|_{C^3([-2\delta,2\delta]\times\mathbb{T})}\|g\|_{H^{3}(\mathbb{T})}\|h\|_{H^{3}(\mathbb{T})}.
\end{align*}
When $j_3=3$ or $j_3=2$, it can treated similarly as $j_2=3$ and $j_2=2$.
\end{proof}
\section*{Acknowledgements}
 The author sincerely thanks Charles Fefferman for introducing this problem and for all the helpful discussions. The author also gratefully thanks Javier Gomez-Serrano and Jaemin Park for useful discussions on Section \ref{analyticitystationary}. This material is based upon work while the author studied at Princeton University. JS was partially supported by NSF through Grant NSF DMS-1700180 and by the European Research Council through ERC-StG-852741-CAPA.
\bibliographystyle{abbrv}
\bibliography{references}

\begin{tabular}{l}
\textbf{Jia Shi}\\
{Department of Mathematics}\\
{Massachusetts Institute of Technology} \\
{Simons Building (Building 2), Room 157}\\
{Cambridge, MA 02139, USA}\\
{e-mail: jiashi@mit.edu}\\ \\
\end{tabular}
\end{document}